\documentclass[12pt,a4paper,intlimits,sumlimits]{amsart}

\usepackage{amsmath, amsthm, mathtools}
\usepackage{amsfonts}
\usepackage[alphabetic, nobysame]{amsrefs}
\usepackage{graphicx}
\usepackage{framed}
\usepackage{amssymb}
\usepackage{esvect}
\usepackage{xcolor}
\usepackage{eucal}
\usepackage{mathrsfs}
\usepackage{hyperref}
\usepackage[english]{babel}
\usepackage{mathrsfs}
\usepackage{esint}

\def \N {\mathbb{N}}
\def \R {\mathbb{R}}

\theoremstyle{definition}
\newtheorem{definition}{Definition}[section]

\theoremstyle{plain}
\newtheorem{theorem}[definition]{Theorem}
\newtheorem{proposition}[definition]{Proposition}
\newtheorem{lemma}[definition]{Lemma}

\numberwithin{equation}{section}

\usepackage{geometry}
\renewcommand{\epsilon}{\varepsilon}
\newcommand{\e}{\varepsilon}

\renewcommand{\leq}{\leqslant}
\renewcommand{\le}{\leqslant}
\renewcommand{\geq}{\geqslant}
\renewcommand{\ge}{\geqslant}

\allowdisplaybreaks

 \title[Logistic diffusion equations for the superposition of operators]{Logistic diffusion equations \\ governed by the superposition\\ of operators of mixed fractional order}

\author[S. Dipierro, E. Proietti Lippi, C. Sportelli and E. Valdinoci]{Serena Dipierro, Edoardo Proietti Lippi, Caterina Sportelli and Enrico Valdinoci}

\address{Department of Mathematics and Statistics
\newline\indent University of Western Australia \newline\indent
35 Stirling Highway, WA 6009 Crawley, Australia.\newline
\newline\indent
\tt serena.dipierro@uwa.edu.au \newline\indent
\tt edoardo.proiettilippi@uwa.edu.au \newline\indent
\tt caterina.sportelli@uwa.edu.au \newline\indent
\tt enrico.valdinoci@uwa.edu.au}

\begin{document}

\maketitle

\begin{abstract}
We discuss the existence of stationary solutions for logistic
diffusion equations of
Fisher-Kolmogoroff-Petrovski-Piskunov type
driven by the superposition of fractional operators in a bounded region with ``hostile'' environmental conditions, modeled by homogeneous external Dirichlet data. 

We provide a range of results on the existence and nonexistence of solutions tied to the spectral properties of the ambient space, corresponding to either survival or extinction of the population.

We also discuss how the possible presence of nonlocal phenomena of concentration and diffusion
affect the endurance or disappearance of the population.

In particular,  we give examples in which
both classical and anomalous diffusion leads to the extinction of the species,
while the presence of an arbitrarily small concentration pattern enables survival.
\end{abstract}

\tableofcontents

\section{Introduction}\label{intro}
In this paper we investigate the existence of stationary solutions of logistic  diffusion equations of
Fisher-Kolmogoroff-Petrovski-Piskunov type involving the continuous superposition of fractional operators of different orders under Dirichlet external conditions. 

{F}rom the biological perspective, the logistic diffusion equation will account for stationary equilibria in terms of the density~$u$ of a certain biological population in competition for environmental resources (and possibly subject to pollination).

The environment considered here corresponds to a safe location (that will be modeled by an open bounded set~$\Omega$ of~$\R^N$) surrounded by a hostile habitat (corresponding to a homogeneous Dirichlet condition for the population density~$u$ in the complement of~$\Omega$). In practice, this kind of hazardous environment occurs frequently in nature, e.g. due to the high abundance of natural enemies of a given biological species, such as parasites and predators. 

The main focus of this paper concerns the persistence or extinction of the given species. Of course, in practice many complex factors determine the chances of survival of a population and one of these factors is the dispersal pattern of the species. In this paper, we account for a composite diffusive mode, in which different individuals can follow random dispersal strategies corresponding to either Gaussian or L\'evy distributions. Moreover, differently from the existing literature, we allow the dispersive modes corresponding to lower L\'evy exponents to take place with the ``opposite sign''.

These features offer significant mathematical challenges and provide valuable biological consequences. Indeed, while the species' movement in the biological literature is usually described as ``directed'' (individuals move from high to low population densities to avoid crowding or to seek environmental resources, see e.g.~\cite{MR707172}, and this corresponds to a parabolic equation in the mathematical framework), we also allow a proportion of individuals to perform an ``inverse'' diffusion (whose mathematical counterpart is a parabolic equation with a time inversion). In this sense, while parabolic equations correspond to the dispersal of an initially concentrated mass, the inversion of time models congregation and clustering phenomena. It is indeed not unreasonable to imagine that gathering may be beneficial for population survival in the presence of a lethal environment surrounding a secure niche, and one of the outcomes of this article is to provide a sound mathematical ground for this intuition.

To allow for a rigorous mathematical formulation, the superposition of the dispersals is modeled by a measure~$\mu$
over the diffusive exponent~$s\in[0,1]$: roughly speaking, the mass assigned by this measure to each exponent~$s$
is proportional to the number of individuals adopting a diffusive strategy with exponent~$s$ (but our model is general enough to allow
both discrete and continuous distributions).

We point out that low and high values of the exponent~$s$ play a somewhat different role in biological dispersal,
according to the different scales taking place in the corresponding parabolic equation (see e.g.~\cite{MR4671814}).
Namely, if~$s=1$ corresponds to the classical Gaussian dispersal, diffusive patterns with exponent~$s$ close to~$1$
correspond to quite regular and localized motions, while the ones related to exponents~$s$ close to~$0$
represent more ``sudden'' movement (since in this situation the corresponding equation can both maintain
a considerable portion of the initial mass in its initial position and sudden deliver another 
considerable portion of the mass far away in space).

Interestingly, in our model we can allow the measure~$\mu$ to be, in fact, a signed measure,
with a positive sign corresponding to high values of~$s$ and possibly negative sign
related to low values of~$s$. This allows for a portion of the population to exhibit quick concentration
phenomena in the safe niche, and we will see that this feature can provide a useful defensive mechanism for the survival of the species.\medskip

We also observe a parallel between the superposition of diffusive operators of different orders and the dimensional analysis of various scales in fluid dynamics. In the fluid dynamic context, the viscous effect, modeled by the Laplace operator, is primarily considered to occur at ``small scales''. At these scales, large vortices in the inertial range remain mostly unaffected by viscosity, allowing energy transfer across scales with minimal dissipation.

In a similar vein, the simultaneous presence of operators of different orders suggests that they may act at distinct scales. For instance, an operator of order~$2s$, corresponding to a Fourier symbol such as~$|\xi|^{2s}$(where~$\xi$ represents a frequency, dimensionally the inverse of a length), is expected to dominate only when~$|\xi|^{2s}$ exceeds a certain threshold~$T$. This dominance implies influence primarily at length scales smaller than~$T^{-\frac1{2s}}$. Notably, this spatial threshold approaches a universal value (of~$1$
in this system of measurement) for small values of the operator of order~$s$, suggesting that superposition operators manifest their effects across various scales, with ballistic diffusion possibly providing a kind of universal pattern.

However, it is important to approach this perspective cautiously, as even our present understanding of fluid dynamics still remains incomplete. Nonetheless, considerations like these can foster useful heuristic insights that, with time and effort, may lead to the development of rigorous theories.
\medskip

The results obtained in this paper can be summarized as follows.
First (as shown in Theorem~\ref{thmlogistica2}
below), we determine the survival chances of a biological species based on environmental resources. In simple terms, if resources are too scarce, the only solution to the stationary logistic equation is the trivial one, meaning there is no equilibrium population but only extinction. On the other hand, if resources are sufficiently abundant, the population survives even in the hostile environment, and the logistic diffusion
equation has a nonnegative, nontrivial
(that is, not identically zero) stationary solution.

These findings align with our intuitive understanding of the model. However, we also provide a precise threshold for survival and extinction, based on the principal eigenvalue of the domain. This approach is not entirely standard both because, in our case, the diffusive operator is not of fixed order, and because the associated Dirichlet form may have negative contributions. Thus, to derive this result, we first need to analyze the linearized equation and formalize the concept of the principal eigenvalue in this context (see Theorem~\ref{propabc}).

Additionally, we examine the role of the negative component of the Dirichlet form
and relate it to the chances of survival of the population. Specifically, as shown in Theorem~\ref{thmlogistica3}, we present examples where this negative component significantly benefits the species' survival. We demonstrate, with quantitative bounds, that in some cases the positive contributions to the Dirichlet form (associated with either classical or anomalous diffusion) lead to extinction, and yet even an arbitrarily small negative component can enable survival. This result has interesting biological implications, suggesting that concentration phenomena (modeled by the negative energy contribution) may promote subsistence. Indeed, while diffusion facilitates exploration of the environment and resource exploitation, it can also lead to the death of individuals wandering into harmful areas. Concentration, by contrast, tends to keep a significant portion of the species away from lethal zones.

In this spirit, we recall that the idea that the survival of a species relies on a balance reached between the homogenizing influence of diffusion and the aspects of population interactions is not new in the biological literature and it is often related to pattern formation (see e.g.~\cite{LEVIN}), but we think that our approach in terms of forward and backward in time diffusion 
(with the latter providing concentrating effects) is new and that the quantitative results obtained here can foster new interdisciplinary interactions.

We also explore the role of anomalous diffusion in survival, both with and without concentration phenomena. Specifically (as shown in Theorem~\ref{thmlogistica4}), we provide examples in which the safe niche
consists of two equal and disjoint regions which do not allow separately for the survival of the species,
but their combination, in tandem with a nonlocal dispersal, permits the persistence of the species. This
is an interesting feature, highlighting the beneficial role of L\'evy flights in biological scenario,
and we show that this phenomenon takes place both in the presence and in the absence of
the concentration properties modeled by the possibly negative contributions of the measure.

It is also interesting to relate diffusive properties of the population
and the size of the survival niche. We will show (see Theorem~\ref{thmduemisure})
that when the safe region is suitably small
the species with small diffusive exponents may have a survival advantage, while
in large regions larger diffusive exponents become favorable. This phenomenon is not completely
intuitive and, possibly, it is related to the fact that
to survive in small niches in hostile habitats the best is to preserve a portion of the population
inside the safe domain, which is guaranteed by the small L\'evy exponents (albeit of small size, instead,
the wanderings provided by the Gaussian diffusion may risk to set the population in the hostile territory).
Instead, in large domains, the fact that strong L\'evy flights push a considerable proportion
of the species in the lethal area ``close to infinity'' makes it plausible to seek
advantage in smaller scale diffusion patterns.\medskip

Let us now dive into some mathematical details of our model.
The continuous superposition of (possibly) fractional operators occurs through operators of the form
\begin{equation}\label{superpositionoperator}
L_\mu u:=\int_{[0, 1]} (-\Delta)^s u\, d\mu(s),
\end{equation}
defined via a signed measure
\begin{equation}\label{misure+-}
\mu:=\mu^+ -\mu^-.
\end{equation}
Here~$\mu^+$ and~$\mu^-$ denote two nonnegative finite Borel measures over the fractional range~$[0, 1]$. 

As customary, the notation~$(-\Delta)^s$ is reserved to the fractional Laplacian, defined, for all~$s\in(0,1)$, as
\begin{equation}\label{AMMU0} (- \Delta)^s\, u(x): = c_{N,s}\int_{\R^N} \frac{2u(x) - u(x+y)-u(x-y)}{|y|^{N+2s}}\, dy.
\end{equation}
The positive normalizing constant~$c_{N,s}$
is chosen in such a way that, for~$u$ smooth and rapidly decaying, the Fourier transform of~$(-\Delta)^s u$ returns~$(2\pi|\xi|)^{2s}$ times the Fourier transform of~$u$ and provides consistent limits as~$s\nearrow1$ and as~$s\searrow0$, namely
\[
\lim_{s\nearrow1}(-\Delta)^su=(-\Delta)^1u=-\Delta u
\qquad{\mbox{and}}\qquad
\lim_{s\searrow0}(-\Delta)^s u=(-\Delta)^0u=u,
\]
see e.g.~\cite{MR2944369} (actually, the precise value of this constant, which will be recalled in~\eqref{defcNs}, does not play
a major role here and one can consider it just a convenient normalization to
allow the limit cases~$s=0$ and~$s=1$ to be included in a unified notation).
\medskip

The notation in~\eqref{superpositionoperator} is conveniently general to allow the simultaneous
treatment of several cases of interest in a unified way. Indeed, when~$\mu$ is the
Dirac measure at~$1$, the operator in~\eqref{superpositionoperator} boils down to (minus) the classical Laplacian, and, similarly, when~$\mu$ is the
Dirac measure at~$s\in(0, 1)$, it reduces to~$(-\Delta)^s$.

As such, the notation in~\eqref{superpositionoperator} allows us to consider also superposition operators, such as~$-\Delta+(-\Delta)^s$, or~$(-\Delta)^{s_1}+(-\Delta)^{s_2}$, with~$s$, $s_1$, $s_2\in(0,1)$, and even
more generally
\begin{equation}\label{aodjfnvhsnd-01}\sum_{k=1}^K c_k\,(-\Delta)^{s_k},\end{equation}
with~$c_k\in(0,+\infty)$ and~$s_k\in[0,1]$.

Actually, even infinite series can be taken into account (provided that they are convergent).

Also the case of continuous superposition can be taken into account, when the measure
is absolutely continuous with respect to the Lebesgue measure in~$[0,1]$. Namely, the choice~$d\mu:= \phi(s)\,ds$ in~\eqref{superpositionoperator}, for some nonnegative and integrable function~$\phi$, returns
\begin{equation}\label{aodjfnvhsnd-02}\int_0^1 (-\Delta)^s u(x)\,\phi(s)\,ds.\end{equation}

Of course, superpositions of~\eqref{aodjfnvhsnd-01} and~\eqref{aodjfnvhsnd-02} are possible as well,
leading to quite complex expressions of the form
$$ \sum_{k=1}^K c_k\,(-\Delta)^{s_k}u(x)+
\int_0^1 (-\Delta)^s u(x)\,\phi(s)\,ds.$$

But, even more interestingly, our setting, as specified in~\eqref{misure+-}, allows for
signed measures. This enables us to comprise, for example, operators of the form~$-\Delta-\e(-\Delta)^s$, or~$(-\Delta)^{s_1}-\e(-\Delta)^{s_2}$, with~$s$, $s_1$, $s_2\in(0,1)$ and~$\e>0$ (that will be taken
sufficiently small in the forthcoming setting), or more generally
$$ \sum_{k=1}^K c_k\,(-\Delta)^{s_k}-\sum_{k=K+1}^{K+M} c_k\,(-\Delta)^{s_k},$$
with~$c_k\in(0,+\infty)$ (and~$c_{K+1},\dots,c_{K+M}$ suitably small)
and~$1\ge s_1>s_2>\dots>s_{K+M}\ge0$.

The continuous superposition in~\eqref{aodjfnvhsnd-02} can also be considered
as a signed measure, for example by supposing that~$\phi\ge0$ in~$[\overline s,1]$
and~$\phi\le0$ in~$[0,\overline s]$, for some~$\overline s\in(0,1]$ (and a sufficiently small contribution from the negative component).

The possible presence of this negative component is interesting from the biological
point of view, modeling the superposition of diffusive and concentration phenomena,
but also challenging in terms of the corresponding mathematical theory, e.g. due to the possible failure of the maximum principle.

{F}rom this perspective, the setting in~\eqref{superpositionoperator} comes in handy, providing
a compact notation to treat simultaneously all the cases described above, and further cases as well,
essentially at no extra cost (in any case, the results that we obtain here are new
for all these specific cases as well).
\medskip

Let us now introduce some specific technical hypotheses that are needed
to provide a solid functional framework to the above setting.
We assume that there exist~$\overline s \in (0, 1]$ and~$\gamma\ge 0$ such that
\begin{equation}\label{mu0}
\mu^+\big([\overline s, 1]\big)>0,
\end{equation}
\begin{equation}\label{mu1}
\mu^-\big([\overline s, 1]\big)=0
\end{equation}
and
\begin{equation}\label{mu2}
\mu^-\big([0, \overline s)\big)\le\gamma\mu^+\big([\overline s, 1]\big).
\end{equation}

Broadly speaking, conditions~\eqref{mu0} and~\eqref{mu1} assert that the components of the signed measure tied to higher fractional exponents are positive, whereas assumption~\eqref{mu2} requires that the negative components (if any) must be appropriately reabsorbed into the positive ones
(as a matter of fact, to realize this, we will choose~$\gamma$ sufficiently small).

Also, we observe that, by~\eqref{mu0}, 
\begin{equation}\label{scritico}
{\mbox{there exists~$s_\sharp \in [\overline s, 1]$ such that }} \mu^+ ([s_\sharp, 1])> 0.
\end{equation}
As described in the forthcoming Proposition~\ref{embeddingsXOmega},  the exponent~$s_\sharp$ plays the role of a critical exponent. Therefore, roughly speaking, while some arbitrariness is allowed in the choice of~$s_\sharp$ here above,
the results obtained will be stronger if one takes~$s_\sharp$ ``as large as possible'', but still satisfying condition~\eqref{scritico} (that is, in all results
it will be possible to simply choose~$s_\sharp :=\overline s$, but the possibility of choosing
a larger~$s_\sharp$, when available, leads to improved quantitative bounds).

This general setting was recently presented in~\cite{MR4736013} and it has been also extended to the case of superposition of fractional~$p$-Laplacians, for a given exponent~$p$ (see~\cite{DPSV2} and also~\cite{MR4793906}), and also to the superposition of fractional operators under Neumann boundary conditions (see~\cite{TUTTI, TUTTI2}). See also~\cite{TUTTI3}, where the 
superposition is made over fractional~$p$-Laplacians of different orders, namely for~$s\in[0,1]$ and~$p\in (1,N)$.
\medskip

We now introduce each result individually and thoroughly.

\subsection{The eigenvalue problem}
As a preliminary topic, needed to study the corresponding nonlinear
logistic equation of
Fisher-Kolmogoroff-Petrovski-Piskunov type,
we address the study of the Dirichlet eigenvalue problem of the operator~$L_\mu$. That is, we let~$\lambda\in\R$ and we take into account the problem
\begin{equation}\label{eigenvalueproblem}
\begin{cases}
L_\mu u= \lambda u &\mbox{ in } \Omega,\\
u=0 &\mbox{ in } \R^N\setminus\Omega.
\end{cases}
\end{equation}
We point out that, in our analysis, due to the presence of possible contributions coming from the negative components of the measure~$\mu$ in~\eqref{misure+-},  the study of~\eqref{eigenvalueproblem} is quite challenging. However,  we are able to prove that problem~\eqref{eigenvalueproblem} possesses a nontrivial solution, corresponding to the first eigenvalue of~$L_\mu$.

Also, if~\eqref{mu2} is paired with an additional assumption, then we can conclude that the first eigenvalue is simple and it admits a nonnegative eigenfunction.

In more detail, we consider the following additional condition.
We take~$\Omega$ to be an open, bounded subset
of~$\R^N$ with Lipschitz boundary and we let~$R>0$ be such that~$\Omega \subset B_R$.
We pick~$\overline\gamma>0$, to be assumed conveniently small in dependence of~$N$, $R$ and~$\overline s$, and we suppose that there exists~$\delta\in (0, 1-\overline{s}]$ such that
\begin{equation}\label{mu2forte}
\mu^-\big((0, \overline s)\big)\le\overline\gamma\,\delta\,\mu^+\big([\overline s, 1-\delta]\big).\end{equation}
A quantitative value for~$\overline\gamma$ can be explicitly determined (see~\eqref{overlinegamma}), but
its precise value does not alter the qualitative features that we are now discussing.

Roughly speaking, in condition~\eqref{mu2forte}
we require that the negative components of the signed measure~$\mu$ on~$[0,\overline{s})$ (if any) must be properly ``reabsorbed" into the positive contributions in~$[\overline{s},1-\delta]$
(i.e., away from the contribution coming from the Laplacian). 
In a nutshell, this hypothesis requires, in a quantitative way, that the negative component of the measure is reabsorbed into a positive component of nonlocal type (we think it is an interesting open problem to decide whether condition~\eqref{mu2forte} can be dropped). 

Also, we point out that when~$\overline\gamma\delta\le\gamma$, condition~\eqref{mu2forte} entails~\eqref{mu2}.
Conversely, \eqref{mu2} does not imply~\eqref{mu2forte} since the contribution on the right-hand side of~\eqref{mu2} could come only for the local component (namely, if~$\mu^+$ is a singular measure concentrated at~$s=1$).

In Appendix~\ref{appendixesempio},
we will present an explicit computation to show how
assumptions~\eqref{mu2} and~\eqref{mu2forte} can be checked in a
concrete case. 

Moreover, we will discuss the need of the additional assumption~\eqref{mu2forte} in Proposition~\ref{propositionfails} in which we provide an explicit example
of what happens if condition~\eqref{mu2forte} is violated.

In our setting , condition~\eqref{mu2forte} is used to ensure the positivity of the first eigenfunction associated with the equation. 

For this, we define the space~$X(\Omega)$ as the set of functions which belong to~$X(\R^N)$ and vanish outside~$\Omega$
(the precise functional setting will be introduced at the beginning of
Section~\ref{sectionpreliminaries}, see in particular~\eqref{normadefinizione}
and~\eqref{XOmegadefn}).
Moreover, we denote by~$[u]_s$ the Gagliardo semi-norm of~$u$, with the
slight abuse of notation that~$[u]_0=\|u\|_{L^2(\Omega)}$ and~$[u]_1=\|\nabla u\|_{L^2(\Omega)}$ (see
the forthcoming formula~\eqref{seminormgagliardo} for full details).

\begin{theorem}\label{propabc}
Let~$\Omega$ be an open and bounded subset of~$\R^N$ with Lipschitz boundary.

Then, there exists~$\gamma_0>0$, depending only on~$N$ and~$\Omega$,
such that if~$\gamma\in[0,\gamma_0]$ the following statements hold true.

Let~$\mu$ satisfy~\eqref{mu0}, \eqref{mu1} and~\eqref{mu2}. Let~$s_\sharp$ be as in~\eqref{scritico}. 

Problem~\eqref{eigenvalueproblem} admits a positive eigenvalue~$\lambda_\mu(\Omega)$ given by
\begin{equation}\label{primoautovalore}
\lambda_\mu(\Omega):=\min_{u\in X(\Omega)\setminus\{0\}} \dfrac{\displaystyle\int_{[0, 1]} [u]^2_s \, d\mu^+(s) - \displaystyle\int_{[0, \overline s)} [u]^2_s \, d\mu^-(s)}{\displaystyle\int_\Omega |u(x)|^2 \,dx}.
\end{equation}
An eigenfunction~$e_\mu$ corresponding to the eigenvalue~$\lambda_\mu(\Omega)$
attains the minimum in~\eqref{primoautovalore}.

If in addition~$\mu$ satisfies~\eqref{mu2forte}, then
every eigenfunction corresponding to the eigenvalue~$\lambda_\mu(\Omega)$ in~\eqref{primoautovalore} does not change sign and~$\lambda_\mu(\Omega)$ is simple.
\end{theorem}

Notice that
when the operator~$L_\mu$ reduces to~$-\Delta$, problem~\eqref{eigenvalueproblem} is classical (see e.g.~\cite[Chapter~6]{MR1625845} and also, in further generality,
\cite{MR437924}). 
The case~$L_\mu=(-\Delta)^s$, with~$s\in(0,1)$, has also been extensively analyzed in the literature, see e.g.~\cite[Proposition~9]{MR3002745}.

While some of the results presented in Theorem~\ref{propabc}
can be seen as part of the general spectral theory of operators (see
e.g.~\cite[Sections~6.10 and~6.11]{MR3753707} and, more generally,~\cite{MR679140}), we provide here a thorough and self-contained characterization of the part of the theory which is specifically needed for the analysis of the logistic diffusion equation carried out in this work.

\subsection{Logistic equations of
Fisher-Kolmogoroff-Petrovski-Piskunov type}
In the classical logistic diffusion model, the increase in population is thought to be directly related to the availability of environmental resources, with growth influenced by a nonnegative function here denoted by~$\sigma$. Moreover, population extinction occurs as resources become exhausted, represented by a nonnegative function
modeling competition for survival that will be denoted by~$\nu$. It is also assumed that the population disperses randomly. 
These observations led to the study of an 
evolution equation of the form
$$\partial_t u=  \Delta u + (\sigma-\nu u)u$$
and its stationary counterpart
\begin{equation}\label{logisicalaplaciano}
\Delta u + (\sigma-\nu u)u=0,
\end{equation}
see e.g.~\cite{VER45, MP12, PR20}.

A nonlocal version of the equation in~\eqref{logisicalaplaciano}
has been put forth in~\cite{MR3579567}, where the authors
consider the cases of bounded subsets with Dirichlet data and
periodic environments and
the leading operator is the fractional Laplacian. 
This is justified by the fact that nonlocal dispersal strategies have been observed in nature and may be related to optimal hunting strategies and adaptation to the environment, see e.g.~\cite{VAB96, HQD10}
an also~\cite{MR4484924, MR4645663}
for an analysis of the so-called L\'evy flight foraging hypothesis.

The case of Neumann boundary conditions and mixed order
operators of the form~$-\alpha\Delta+\beta(-\Delta)^s$, with~$\alpha$, $\beta\ge0$ and~$\alpha+\beta>0$, has been considered in~\cite{MR4651677}.
See also~\cite{06226233, 06858645}, where the authors consider a
logistic diffusion equation with Neumann boundary conditions
related to the spectral fractional Laplacian
and study best dispersal strategies in connection with the principal eigenvalue
of the operator taken into account.
The role of Allee effects has also been taken into account in the literature, see~\cite{MR4764600}
and the references therein.

In this paper we will focus on a variation of the logistic diffusion equation which allows for the possibility of (possibly) infinite fractional concentration and diffusion processes to occur.
To be more precise,  we consider the stationary logistic equation of
Fisher-Kolmogoroff-Petrovski-Piskunov type
\[
L_\mu u=(\sigma-\nu u)u+\tau (J\ast u),
\]
being~$L_\mu$ the operator introduced in~\eqref{superpositionoperator}. 

We point out that this equation is far more general than~\eqref{logisicalaplaciano}, because it allows for phenomena of concentration and diffusion which can happen at different scales.
This is achievable through the overall versatility of the measure~$\mu$ that we take into account here.

We remark that considering negative signs in diffusive operators goes beyond the mere mathematical interest. Indeed, elliptic operators with negative sign can be useful to model concentration phenomena (e.g., in the heat equation with inverted time direction).  Also, in mathematical biology,  the movement of an organism as a consequence of a chemical stimulus and specifically in dependence of the increasing or decreasing concentration of a particular substance is called ``chemotaxis". A simple model in which a Laplacian ``with the wrong sign" naturally appears is that of a biological species subject to chemotaxis when the chemotactic agent is proportional to the logarithm of the density of population. On this topic, we refer the interested reader to~\cite[Chapter~2]{ZZLIB}.
\medskip

In our setting,  the environment with hostile surroundings is illustrated through the Dirichlet problem
\begin{equation}\label{problemalogistico}
\begin{cases}
L_\mu u=(\sigma-\nu u)u+\tau (J\ast u) &\mbox{in } \Omega,\\
u=0 	&\mbox{in } \R^N \setminus \Omega,\\
u\ge 0 &\mbox{in } \R^N.
\end{cases}
\end{equation}
Here, we take~$s_\sharp$ as in~\eqref{scritico} and we set
\[
2^*_{s_\sharp}:=\begin{cases}
\dfrac{2N}{N-2s_\sharp} &\mbox{ if } N>2s_\sharp\\
+\infty &\mbox{ if } N\le 2s_\sharp.
\end{cases}
\]
Moreover, we will assume that~$\sigma$, $\nu:\Omega\to \R$ and~$J:\R^N\to\R$ are nonnegative functions satisfying
\begin{equation}\label{ipotesilogistica1}
\sigma\in L^m(\Omega)\, \mbox{ for some } m>\frac{2^*_{s_\sharp}}{2^*_{s_\sharp}-2}
\quad \mbox{and} \quad \frac{(\sigma+\tau)^3}{\nu^2}\in L^1(\Omega).
\end{equation}

Furthermore, we suppose that~$\tau\in[ 0,+\infty)$ and that~$J$ satisfies
\begin{equation}\label{proprietàJ}
\int_{\R^N}J(x)\,dx=1  \quad \mbox{and} \quad J(-x)=J(x)\,\mbox{ for any }x\in \R^N.
\end{equation}
In our setting, the parameter~$\tau$ and
the interaction kernel~$J$ describe an additional birth rate,
that is due, for instance, to pollination.

We aim to establish an existence theory for nontrivial solutions of~\eqref{problemalogistico}. Also, we will compare the local and nonlocal dynamics of the population that diffuses according to the operator~$L_\mu$, analyzing their effectiveness in relation to the resources and the set~$\Omega$.

We first provide a general existence result:

\begin{theorem}\label{thmlogistica1}
Let~$\Omega$ be an open and bounded subset of~$\R^N$ with Lipschitz boundary.

Then, there exists~$\gamma_0>0$, depending only on~$N$ and~$\Omega$,
such that if~$\gamma\in[0,\gamma_0]$ the following statement holds true.

Let~$\mu$ satisfy~\eqref{mu0}, \eqref{mu1}, \eqref{mu2} and~\eqref{mu2forte}.
Let~$\sigma$, $ \nu$ and~$ J$ satisfy~\eqref{ipotesilogistica1} and~\eqref{proprietàJ}.

Then, there exists a nonnegative weak solution of~\eqref{problemalogistico}.
\end{theorem}

The study of solutions generated by Theorem~\ref{thmlogistica1} can be enhanced by using the spectral analysis developed in Section~\ref{sectioneigenvalues}.
In the next result, we prove that the existence of nontrivial solutions of~\eqref{problemalogistico} can be characterized in terms of the first
eigenvalue of the operator~$L_\mu$.

In further detail, we show that if the resources (encoded by the function~$\sigma$, the parameter~$\tau$
and the convolution kernel~$J$) happens to be too small when compared to~$\lambda_\mu(\Omega)$ defined in~\eqref{primoautovalore}, then problem~\eqref{problemalogistico} admits only the solution~$u\equiv 0$.

If instead the resources turn out to be sufficiently large, then problem~\eqref{problemalogistico} has a positive solution. This scenario corresponds to the survival of the population.

Our result reads as follows:

\begin{theorem}\label{thmlogistica2}
Let~$\Omega$ be an open and bounded
subset of~$\R^N$ with Lipschitz boundary.

Then, there exists~$\gamma_0>0$, depending only on~$N$ and~$\Omega$,
such that if~$\gamma\in[0,\gamma_0]$ the following statements hold true.

Let~$\mu$ satisfy~\eqref{mu0}, \eqref{mu1}, \eqref{mu2} and~\eqref{mu2forte}. Let~$\sigma$, $ \nu$ and~$ J$ satisfy~\eqref{ipotesilogistica1} and~\eqref{proprietàJ}.
Let~$\lambda_\mu(\Omega)$ be given by~\eqref{primoautovalore}.

If
\begin{equation}\label{supsigma}
\sup_\Omega \sigma +\tau \leq \lambda_\mu(\Omega),
\end{equation}
then problem~\eqref{problemalogistico} admits only the trivial solution.

If instead~$e_\mu$ is as in Theorem~\ref{propabc} and
\begin{equation}\label{infsigma}
\lambda_\mu(\Omega)< \inf_\Omega \sigma +\tau \int_\Omega e_\mu (J\ast e_\mu)\, dx,
\end{equation}
then problem~\eqref{problemalogistico} admits a nonnegative and nontrivial 
solution.
\end{theorem}

In light of the chemotaxis processes previously described, the next result is quite interesting. Indeed, we show that, for some~$\sigma$, $\nu$ and~$\tau$, 
if there are no contributions  coming from the negative component of the signed measure~$\mu$, the whole population gets extinct, while
the presence of negative contributions produces
a nontrivial solution (even if these negative contributions are arbitrarily small).

The precise statement of our result is the following:

\begin{theorem}\label{thmlogistica3}
Let~$\Omega$ be an open and bounded subset
of~$\R^N$ with Lipschitz boundary.

Then, there exists~$\gamma_0>0$, depending only on~$N$ and~$\Omega$,
such that if~$\gamma\in[0,\gamma_0]$ the following statement holds true.

Let~$\mu$ satisfy~\eqref{mu0}, \eqref{mu1}, \eqref{mu2} and~\eqref{mu2forte} and suppose that~$\mu^-\not \equiv 0$. Let~$J$ satisfy~\eqref{proprietàJ}.

Then, there exist~$\tau\in[ 0,+\infty)$ and~$\sigma$, $\nu\in L^\infty(\Omega)$ satisfying~\eqref{ipotesilogistica1} such that the problem
\begin{equation}\label{problemchemotaxis}
\begin{cases}
L_{\mu^+} u=(\sigma-\nu u)u+\tau (J*u)   &\mbox{in } \Omega,
\\
u=0 	&\mbox{in } \R^N \setminus \Omega,
\\
u\geq 0 &\mbox{in } \R^N
\end{cases}
\end{equation}
admits only the trivial solution, while for any~$\varepsilon\in (0,1)$, the problem 
\begin{equation}\label{problemchemotaxis2}
\begin{cases}
L_{\mu^+} u-\varepsilon L_{\mu^-} u=(\sigma-\nu u)u+\tau (J*u)   &\mbox{in } \Omega,
\\
u=0 	&\mbox{in } \R^N \setminus \Omega,
\\
u\geq 0 &\mbox{in } \R^N
\end{cases}
\end{equation}
 admits a nontrivial solution.
\end{theorem}

We now consider the scenario where there are two disjoint sets~$\Omega_1$ and~$\Omega_2$ such that the resource in each single set is not sufficient for the species to survive. In spite of this, we show that the combined resources in the union of the sets can be conveniently used by a nonlocal population to survive.
We point out that this result was already known in the case of diffusion driven by the fractional Laplacian
(see~\cite[Theorem~1.3]{MR3579567}).
The interesting feature here is that this phenomenon persists in the presence of a general operator~$L_\mu$, both in the cases~$\mu^-\equiv 0$ and~$\mu^-\not\equiv 0$.

This scenario is described by the following result:

\begin{theorem}\label{thmlogistica4}
Let~$\Omega$ be an open and bounded subset of~$\R^N$ with Lipschitz
boundary, and suppose that~$\Omega=\Omega_1\cup\Omega_2$, where~$\Omega_1$ and~$\Omega_2$ are two congruent sets such that~$\overline{\Omega_1}\cap\overline{\Omega_2}=\varnothing$. Let~$R>0$ be such that~$\Omega \subset B_R$.

Then, there exists~$\gamma_0>0$, depending only on~$N$ and~$\Omega$,
such that if~$\gamma\in[0,\gamma_0]$ the following statement holds true.

Let~$\mu$ satisfy~\eqref{mu0}, \eqref{mu1}, \eqref{mu2} and~\eqref{mu2forte}
and assume that~$\mu^+\big((0,1)\big)>0$.
Let~$J$ satisfy~\eqref{proprietàJ}.

Then, there exist~$\tau\ge 0$ and~$\sigma$, $\nu\in L^\infty(\Omega)$ satisfying~\eqref{ipotesilogistica1} such that, for any~$i\in\{1,2\}$, the problem
\begin{equation}\label{problemai}
\begin{cases}
L_\mu u=(\sigma-\nu u)u+\tau (J*u)   &\mbox{in } \Omega_i,\\
u=0 	&\mbox{in } \R^N \setminus \Omega_i,\\
u\geq 0 &\mbox{in } \R^N.
\end{cases}
\end{equation}
admits only the trivial solution, while problem~\eqref{problemalogistico} admits a nontrivial solution.
\end{theorem}

In light of Theorem~\ref{thmlogistica2},
we have that the survival of the population is influenced by the
eigenvalue of the operator~$L_\mu$ in a set~$\Omega$,
and therefore by the ``size" of~$\Omega$.

In order to investigate further this topic, for any~$r>0$ we define
\[
\Omega_r:=r\Omega=\big\{ rx \;{\mbox{ with }}\; x\in\Omega\big\}.
\]
We present the following two results:

\begin{proposition}\label{proplogistica5}
Let~$\Omega$ be an open and bounded subset of~$\R^N$ with Lipschitz boundary.

Then, there exists~$\gamma_0>0$, depending only on~$N$ and~$\Omega$,
such that if~$\gamma\in[0,\gamma_0]$ the following statement holds true.

Let~$\mu$ satisfy~\eqref{mu0}, \eqref{mu1}, \eqref{mu2} and~\eqref{mu2forte}.
Let~$\tau\in[0,+\infty)$ and let~$J$ satisfy~\eqref{proprietàJ}.

Let~$\Sigma := \mbox{supp}(\mu^+)$ and assume that
\begin{equation}\label{ipotesilambdamupiuPRE} 
\lambda_{\mu^+} (\Omega)\le \inf_{s\in\Sigma} r^{2s}.\end{equation}
Also, assume that 
\begin{equation}\label{ipotesilambdamupiu}
{\mbox{either~$\mu^-\not\equiv0$ or}} \quad
\lambda_{\mu^+} (\Omega)<\inf_{s\in\Sigma} r^{2s}.
\end{equation}

Then,
\[
\begin{cases}
L_\mu u = (1-u)u + \tau(J\ast u) &\mbox{ in } \Omega_r,\\
u=0 &\mbox{ in } \R^N\setminus\Omega_r,\\
u\ge 0 &\mbox{ in } \R^N
\end{cases}
\]
admits a nontrivial solution.
\end{proposition}

\begin{theorem}\label{thmduemisure}
Let~$\Omega$ be an open and bounded subset of~$\R^N$ with Lipschitz boundary. 

Then, there exists~$\gamma_0>0$, depending only on~$N$ and~$\Omega$,
such that if~$\gamma\in[0,\gamma_0]$ the following statement holds true.

Let~$\mu_1$ and~$\mu_2$ be two positive measures satisfying~\eqref{mu0},
\eqref{mu1} and~\eqref{mu2}, and such that 
\[
\sup\{\mbox{supp}(\mu_1)\}<\inf\{\mbox{supp}(\mu_2)\}. 
\]
Let~$s_\sharp(\mu_1)$ and~$s_\sharp(\mu_2)$ as in~\eqref{scritico} respectively\footnote{We stress that, according to~\eqref{scritico}, $s_\sharp(\mu_1)$ and~$s_\sharp(\mu_2)$ are such that~$\mu_1([s_\sharp(\mu_1),1])>0$ and~$\mu_2([s_\sharp(\mu_2),1])>0$.} for~$\mu_1$ and~$\mu_2$. Let~$J$ satisfy~\eqref{proprietàJ}.

Then,
there exists~$\underline{r}>0$ such that, for any~$r\in (0,\underline{r})$, there exist~$\tau_r \ge 0$ and~$\sigma_r$, $\nu_r\in L^\infty(\Omega)$ satisfying~\eqref{ipotesilogistica1} such that the problem
\begin{equation}\label{problemaduemisure1}
\begin{cases}
L_{\mu_1} u = (\sigma_r-\nu_r u)u + \tau_r(J\ast u) &\mbox{ in } \Omega_r,\\
u=0 &\mbox{ in } \R^N\setminus\Omega_r,\\
u\ge 0 &\mbox{ in } \R^N
\end{cases}
\end{equation}
admits a nontrivial solution, while the problem
\begin{equation}\label{problemaduemisure2}
\begin{cases}
L_{\mu_2} u = (\sigma_r-\nu_r u)u + \tau_r(J\ast u) &\mbox{ in } \Omega_r,\\
u=0 &\mbox{ in } \R^N\setminus\Omega_r,\\
u\ge 0 &\mbox{ in } \R^N
\end{cases}
\end{equation}
admits only the trivial solution.

Moreover, there exists~$\overline{r}\geq \underline{r}$ such that, for any~$r\in (\overline{r},+\infty)$, there exist~$\tau_r \ge 0$ and~$\sigma_r$, $\nu_r\in L^\infty(\Omega)$ satisfying~\eqref{ipotesilogistica1} such that problem~\eqref{problemaduemisure1} admits only the trivial 
solution, while problem~\eqref{problemaduemisure2} admits a nontrivial solution.
\end{theorem}

Theorem~\ref{thmduemisure} says that
in a ``small" domain~$\Omega$
the ``strongly nonlocal'' diffusive species, corresponding to a measure~$\mu_1$ supported in a set consisting of small values of~$s$, may be favored
with respect to a diffusive species related to a measure~$\mu_2$ supported
in a set consisting of larger values of~$s$.
In ``large" domains, instead, a ``mildly nonlocal'' or local diffusion may favor
survival while  a``strongly nonlocal'' diffusion may lead to extinction.

\subsection{Organization of the paper}
The paper is organized as follows. In Section~\ref{sectionpreliminaries} we present the functional framework needed in our setting and we provide some preliminary results.

Section~\ref{sectioneigenvalues} contains the study of the first eigenvalue of problem~\eqref{eigenvalueproblem}.

Section~\ref{sectionlogistica} is devoted to the study of the logistic diffusion equation~\eqref{problemalogistico} and contains the proofs of Theorems~\ref{thmlogistica1}, \ref{thmlogistica2}, \ref{thmlogistica3}, \ref{thmlogistica4} and~\ref{thmduemisure} and Proposition~\ref{proplogistica5}.

Appendix~\ref{appendixesempio} contains 
an explicit computation to show how to check that
the conditions~\eqref{mu2} and~\eqref{mu2forte}
are satisfied in a concrete case.

\section{Preliminary results}\label{sectionpreliminaries}

In this section we introduce the functional analytical setting needed to address the problems introduced in the previous section
and we gather some preliminary observations.

To this end, for~$s\in[0,1]$, we let
\begin{equation}\label{seminormgagliardo}
[u]_s:=
\begin{cases}
\|u\|_{L^2(\R^N)}  &\mbox{ if } s=0,
\\ \\
\displaystyle\left(c_{N,s}\iint_{\R^{2N}}\frac{|u(x)-u(y)|^2}{|x-y|^{N+2s}}\,dx\,dy \right)^{1/2} &\mbox{ if } s\in(0,1),
\\ \\
\|\nabla u\|_{L^2(\R^N)}  &\mbox{ if } s=1.
\end{cases}
\end{equation}
Here above,
\begin{equation}\label{defcNs}
c_{N,s}:=\frac{2^{2s-1}\,\Gamma\left(\frac{N+2s}{2}\right)}{\pi^{N/2}\,\Gamma(2-s)}\, s(1-s),
\end{equation}
being~$\Gamma$ the Gamma function.

We define the space~$X(\R^N)$ as the completion of~$C^\infty_c (\R^N)$ with respect to the seminorm
\begin{equation}\label{normadefinizione}
\|u\|_{X}:=\left(\;\int_{[0,1]}[u]^2_s\,d\mu^+(s)\right)^{\frac12}.
\end{equation}
Also, for any open and bounded set~$\Omega\subset\R^N$
with Lipschitz boundary, we define the space 
\begin{equation}\label{XOmegadefn}
X(\Omega):=\big\{u\in X(\R^N)\;\mbox{ s.t. }\;
u\equiv 0 \mbox{ in }\R^N\setminus \Omega \big\}.
\end{equation}
We remark that
\begin{equation}\label{XHilbert}
\mbox{$X(\Omega)$ is a Hilbert space with respect to~$\|u\|_{X}$}
\end{equation}
and is endowed with the scalar product\footnote{We point out that we write
\begin{equation}\label{ocmrgexpresfyrj5784i}
\int_{[0,1]}c_{N,s}\iint_{\R^{2N}}\frac{(u(x)-u(y))(v(x)-v(y))}{|x-y|^{N+2s}}\,dx\,dy\,d\mu^+(s)\end{equation}
with an abuse of notation. Indeed, to be precise, one should write \label{footimpr2}
\begin{eqnarray*}&&
\int_{(0,1)}c_{N,s}\iint_{\R^{2N}}\frac{(u(x)-u(y))(v(x)-v(y))}{|x-y|^{N+2s}}\,dx\,dy\,d\mu^+(s)
\\&&\qquad\qquad+\mu^+(\{0\})\int_{\Omega}u(x)v(x)\,dx
+\mu^+(\{1\})\int_{\Omega}\nabla u(x)\cdot \nabla v(x)\,dx.
\end{eqnarray*}
To ease notation, unless otherwise specified,
we will always use the compact expression~\eqref{ocmrgexpresfyrj5784i}.}  defined, for any~$u$, $v\in X(\Omega)$, as
\begin{equation}\label{scalarepiu}
\langle u, v\rangle_+ :=  \int_{[0,1]}c_{N,s}\iint_{\R^{2N}}\frac{(u(x)-u(y))(v(x)-v(y))}{|x-y|^{N+2s}}\,dx\,dy\,d\mu^+(s).
\end{equation}

For future reference, recalling hypothesis~\eqref{mu1}, we also define
the scalar product
\begin{equation}\label{scalaremeno}
\langle u, v\rangle_- :=  \int_{[0,\overline s)}c_{N,s}\iint_{\R^{2N}}\frac{(u(x)-u(y))(v(x)-v(y))}{|x-y|^{N+2s}}\,dx\,dy\,d\mu^-(s).
\end{equation}

In our analysis, we will make use of the following result of Sobolev type, which states that higher exponents in fractional norms control the lower exponents with uniform constants (see~\cite[Lemma~2.1]{MR4736013}).
\begin{lemma}\label{nons} 
Let~$0\le s_1 \le s_2 \le1$. 

Then, for any measurable function~$u:\R^N\to\R$ with~$u=0$ a.e. in~$\R^N\setminus\Omega$ we have that
\begin{equation}\label{spp}
[u]_{s_1}\le C \, [u]_{s_2},
\end{equation}
for a suitable positive constant~$C=C(N,\Omega)$.
\end{lemma}

We also observe that, in this setting, we can ``reabsorb'' the negative part of the signed measure~$\mu$,  as proved in~\cite[Proposition~2.3]{MR4736013}.

\begin{proposition}\label{crucial} 
Let~$\mu$ satisfy~\eqref{mu0}, \eqref{mu1} and~\eqref{mu2}.

Then, there exists~$c_0=c_0(N,\Omega)>0$ such that, for any~$u\in X(\Omega)$,
\[
\int_{[{{ 0 }}, \overline s)} [u]_{s}^2 \, d\mu^- (s) \le c_0\,\gamma \int_{[\overline s, 1]} [u]^2_{s} \, d\mu(s).
\]
\end{proposition}

In addition, we recall the result stated in~\cite[Proposition~2.4]{MR4736013}, which provides some useful embeddings for the space~$X(\Omega)$ defined in~\eqref{XOmegadefn}.

\begin{proposition}\label{embeddingsXOmega}
Let~$\mu$ satisfy~\eqref{mu0}, \eqref{mu1} and~\eqref{mu2}.
Let~$s_\sharp$ be as in~\eqref{scritico}.

Then, the space~$X(\Omega)$ is continuously embedded in~$H^{s_\sharp}(\Omega)$.

Furthermore, if~$N>2s_\sharp$, then the space~$X(\Omega)$ is continuously embedded in~$L^r(\Omega)$ for any~$r\in [1,2^*_{s_\sharp}]$, and if~$N\le 2s_\sharp$, then the space~$X(\Omega)$ is continuously embedded in~$L^r(\Omega)$ for any~$r\in [1,+\infty)$.

In addition, the space~$X(\Omega)$ is compactly embedded in~$L^r(\Omega)$ for any~$r\in [1,2^*_{s_\sharp})$.
\end{proposition}

We point out that the assumption on~$\mu$
in~\eqref{mu2} allows us to provide a proper reabsorbing property of the negative component of the measure (if any) and some suitable embeddings for the space~$X(\Omega)$, but it is not strong enough to let us carry on the study of the spectral properties of the operator~$L_\mu$ in~\eqref{superpositionoperator} and the existence of nonnegative solutions of problem~\eqref{problemalogistico}.

For this reason, one needs to introduce a further assumption, which is the one stated in~\eqref{mu2forte}. 

To start with, we observe that the constant~$c_{N, s}$ defined in~\eqref{defcNs} is uniformly bounded for any~$s\in (0,1)$.
In the next result we show an explicit upper bound for the constant~$c_{N, s}$, which holds for any~$s\in (0, 1)$. We provide an explicit lower bound as well, provided that~$s$ is detached from~$0$ and~$1$.

\begin{lemma}\label{lemmacns}
Let~$c_{N, s}$ be as in~\eqref{defcNs}. 
Let also
\begin{equation}\label{gammas}
\overline{\Gamma}_N:=\max_{s\in [0,1]}\dfrac{\displaystyle\Gamma\left(\dfrac{N+2s}{2}\right)}{\displaystyle\Gamma(2-s)}\qquad{\mbox{and}}\qquad
\underline{\Gamma}_N:=
\min_{s\in [0,1]}\dfrac{\displaystyle\Gamma\left(\dfrac{N+2s}{2}\right)}{\displaystyle\Gamma(2-s)}.\end{equation}
Then,
\begin{equation}\label{UPPERBCNS}
c_{N, s}\le\overline{c}_N:= \frac{2\overline{\Gamma}_N}{\pi^{N/2}}\in(0,+\infty).
\end{equation}

Moreover, let~$\overline s$ be as in~\eqref{mu0} and~$\delta\in (0, 1-\overline s]$. If~$s\in [\overline s, 1-\delta]$, then
\begin{equation}\label{LOWBCNS}
c_{N, s}\ge\underline{c}_N\,\delta,\qquad{\mbox{where}}\qquad \underline{c}_N:= \frac{\underline{\Gamma}_N\,{\overline{s}}}{2\pi^{N/2}}\in(0,+\infty).
\end{equation}
\end{lemma}

\begin{proof}
We remark that the Gamma function~$\Gamma (t)$ is continuous and~$\Gamma(t)>0$ for any~$t>0$, so that
\begin{eqnarray*}&&0<\min_{s\in [0,1]}\Gamma\left(\dfrac{N+2s}{2}\right)\le
\max_{s\in [0,1]}\Gamma\left(\dfrac{N+2s}{2}\right)<+\infty\\
\mbox{and} &&0<\min_{s\in [0,1]}\Gamma(2-s)\le \max_{s\in [0,1]}\Gamma(2-s)<+\infty.
\end{eqnarray*}
Accordingly, we have that~$\overline{c}_N$, $\underline{c}_N\in(0,+\infty)$.

Now, the upper bound in~\eqref{UPPERBCNS} and, when~$s\in [\overline s, 1-\delta]$, the lower bound in~\eqref{LOWBCNS}
plainly follow from the definition in~\eqref{defcNs}.
\end{proof}

In light of Lemma~\ref{LOWBCNS}, we can take~$\overline\gamma$ in~\eqref{mu2forte} as
\begin{equation}\label{overlinegamma}
\overline\gamma\in\left[0, \frac{ \underline\Gamma_N \,\overline{s}}{4\overline{\Gamma}_N\,\max\{1,(2R)^2\}}\right).
\end{equation}

For our purposes, Lemma~\ref{lemmacns} needs to be combined with the following technical observation:

\begin{lemma} \label{c1mumeno.le}
Let~$\overline s$ be as in~\eqref{mu0} and~$S\in[\overline{s},1]$.

Then, for every~$\zeta\in(0,2R)$,
$$ \min\{1,(2R)^{2\overline{s}}\}\,\mu^+\big( [\overline{s},S)\big)\int_{(0,\overline{s})}\frac{d\mu^-(s)}{\zeta^{2s}}\le \max\{(2R)^{2\overline{s}},(2R)^{2}\}\,
\mu^-\big( (0,\overline{s})\big)\int_{[\overline{s},S)}\frac{d\mu^+(s)}{\zeta^{2s}}.$$
\end{lemma}

\begin{proof} Let~$s\in(0,\overline{s})$ and~$\sigma\in[\overline{s},S)$.
Since~$\frac{2R}{\zeta}>1$, we have that~$\left(\frac{2R}{\zeta}\right)^{2\sigma}>
\left(\frac{2R}{\zeta}\right)^{2s}$. Integrating this inequality with respect to~$\mu^+$
in~$\sigma\in[\overline{s},S)$ we conclude that
$$ \int_{[\overline{s},S)}\left(\frac{2R}{\zeta}\right)^{2\sigma}\,d\mu^+(\sigma)\ge
\mu^+\big( [\overline{s},S)\big)\left(\frac{2R}{\zeta}\right)^{2s}.$$
We now integrate this inequality with respect to~$\mu^-$ in~$s\in(0,\overline{s})$ and we obtain that
\begin{equation}\label{768594djsagdjkasgfkrtyuijbgr8e6t4y}
\mu^-\big( (0,\overline{s})\big)\int_{[\overline{s},S)}\left(\frac{2R}{\zeta}\right)^{2\sigma}\,d\mu^+(\sigma)\ge
\mu^+\big( [\overline{s},S)\big)\int_{(0,\overline{s})}\left(\frac{2R}{\zeta}\right)^{2s}\,d\mu^-(s).\end{equation}

Moreover, by the monotonicity of the exponential function, we have that, for all~$s\in(0,\overline{s})$ and~$\sigma\in[\overline{s},S)$,
$$ (2R)^{2s}\ge\min\{1,(2R)^{2\overline{s}}\}\qquad
{\mbox{and}}\qquad (2R)^{2\sigma}\le\max\{(2R)^{2\overline{s}},(2R)^{2}\}.$$
{F}rom this and~\eqref{768594djsagdjkasgfkrtyuijbgr8e6t4y},
we obtain the desired result.
\end{proof}

The following result states that, if the reabsorbing property introduced in~\eqref{mu2forte} holds true, then up to replacing~$u$ with~$|u|$, we can decrease the ``energy". This property will play a crucial role in proving that solutions of problem~\eqref{problemalogistico} can be taken to be nonnegative.

\begin{lemma}\label{lemmamodulo}
Let~$\mu$ satisfy~\eqref{mu0} and~\eqref{mu1}
for some~$\overline s\in (0, 1)$.
Let~$R>0$ be such that~$\Omega\subset B_R$ and let~$\delta\in (0, 1-\overline{s}]$. Assume that~\eqref{mu2forte} holds.

Then, for any~$u\in X(\Omega)$,
\begin{equation}\label{crucial.eq} 
\int_{[0, 1]} [|u|]^2_s \, d\mu(s) \le \int_{[0, 1]} [u]^2_s \, d\mu(s).
\end{equation}

Furthermore, the inequality in~\eqref{crucial.eq} is strict unless either~$u\ge0$ or~$u\le0$ a.e. in~$\R^N$.
\end{lemma}

\begin{proof}
Let~$u\in X(\Omega)$. We observe that
$$ \int_{[0, 1]} [u]^2_s \, d\mu^+(s)<+\infty\qquad{\mbox{and}}\qquad
\int_{[0, 1]} [u]^2_s \, d\mu^-(s)<+\infty,$$
due to Proposition~\ref{crucial}.

Consequently, since~$\big| |u(x)|-|u(y)|\big|\le|u(x)-u(y)|$, we also have that
$$ \int_{[0, 1]} [|u|]^2_s \, d\mu^+(s)<+\infty\qquad{\mbox{and}}\qquad
\int_{[0, 1]} [|u|]^2_s \, d\mu^-(s)<+\infty.$$

As a consequence, in light of~\eqref{misure+-} and~\eqref{mu1},
the desired result in~\eqref{crucial.eq} is established if we show that
\begin{equation}\label{crucial.eq.bi}
\int_{[0, 1]} \big([|u|]^2_s - [u]^2_s\big) \,d\mu^+(s) - \int_{[0, \overline s)} \big([|u|]^2_s - [u]^2_s\big) \,d\mu^-(s)\le 0.
\end{equation}

Furthermore, we have that~$[|u|]_0 =\| |u|\|_{L^2(\Omega)}=
\| u\|_{L^2(\Omega)}=[u]_0$ and~$[|u|]_1 =\| |\nabla |u||\|_{L^2(\Omega)}=
\| |\nabla u|\|_{L^2(\Omega)}=[u]_1$. Therefore, the desired result in~\eqref{crucial.eq.bi} boils down to
\begin{equation}\label{crucial.eq.ci}
\int_{(0, 1)} \big([|u|]^2_s - [u]^2_s\big) \,d\mu^+(s) - \int_{(0, \overline s)} \big([|u|]^2_s - [u]^2_s\big) \,d\mu^-(s)\le 0.
\end{equation}
To prove this, we let~$u^\pm:=\max\{ \pm u, 0\}$ and we observe that
\begin{equation}\label{feeeyt9tirhfcs2345678hgbhfndjksjri4yti}\begin{split}&
\big| |u(x)|-|u(y)|\big|^2- |u(x)-u(y)|^2
=-2|u(x)|\,|u(y)|+2u(x)\,u(y)\\&\quad=
-2(u^+(x)+u^-(x))(u^+(y)+u^-(y))+2(u^+(x)-u^-(x))\,(u^+(y)-u^-(y))\\&\quad=-4 u^+(x)u^-(y)-4u^-(x)u^+(y).
\end{split}
\end{equation}
Therefore,
\begin{equation*}
\begin{split}
&\frac14\int_{(0, 1)} \big([|u|]^2_s - [u]^2_s\big) \,d\mu^+(s) - \frac14\int_{(0, \overline s)} \big([|u|]^2_s - [u]^2_s\big) \,d\mu^-(s)\\
&\quad= -\int_{(0, 1)}c_{N, s}\iint_{\R^{2N}}\frac{u^+(x) u^-(y)+ u^-(x)u^+(y)}{|x-y|^{N+2s}}\,dx\, dy\, d\mu^+(s)\\
&\qquad+ \int_{(0, \overline s)}c_{N, s}\iint_{\R^{2N}}\frac{u^+(x) u^-(y)+ u^-(x)u^+(y)}{|x-y|^{N+2s}}\,dx\, dy\, d\mu^-(s)\\&\quad= -\int_{(0, 1)}c_{N, s}\iint_{\Omega\times\Omega}\frac{u^+(x) u^-(y)+ u^-(x)u^+(y)}{|x-y|^{N+2s}}\,dx\, dy\, d\mu^+(s)\\
&\qquad+ \int_{(0, \overline s)}c_{N, s}\iint_{{\Omega\times\Omega}}\frac{u^+(x) u^-(y)+ u^-(x)u^+(y)}{|x-y|^{N+2s}}\,dx\, dy\, d\mu^-(s).
\end{split}
\end{equation*}
Namely,
\begin{equation}\label{crucial.eq.diPRE}
\begin{split}
&\frac14\int_{(0, 1)} \big([|u|]^2_s - [u]^2_s\big) \,d\mu^+(s) - \frac14\int_{(0, \overline s)} \big([|u|]^2_s - [u]^2_s\big) \,d\mu^-(s)\\
&\quad=\iint_{\Omega\times\Omega}\big(u^+(x) u^-(y)+ u^-(x)u^+(y)\big)\\&\qquad\qquad\times\left[-\int_{(0, 1)} \frac{c_{N, s}}{|x-y|^{N+2s}} \,d\mu^+(s) +\int_{(0, \overline s)} \frac{c_{N, s}}{|x-y|^{N+2s}} \,d\mu^-(s)\right] \,dx\, dy
\end{split}
\end{equation}
and accordingly, to prove~\eqref{crucial.eq.ci}, it suffices to show that
\begin{equation}\label{crucial.eq.di}
\int_{(0, \overline s)} \frac{c_{N, s}}{|x-y|^{N+2s}} \,d\mu^-(s) < \int_{(0, 1)} \frac{c_{N, s}}{|x-y|^{N+2s}} \,d\mu^+(s)
.\end{equation}

To prove this, we proceed as follows.
{F}rom Lemma~\ref{lemmacns}, we have that
\begin{equation}\label{c1mumeno}
\begin{split}&
\int_{(0, \overline s)} \frac{c_{N, s}}{|x-y|^{N+2s}} \,d\mu^-(s) \le \frac{\overline{c}_N}{|x-y|^N} \int_{(0, \overline s)}\frac{d\mu^-(s)}{|x-y|^{2 s}}.
\end{split}
\end{equation}
Similarly\footnote{Notice here that we are assuming~$\mu^+\big([\overline s, 1-\delta]\big)>0$, otherwise,
by~\eqref{mu2forte}, we would have~$\mu^-\big( (0,\overline s)\big)=0$
and~\eqref{crucial.eq.bi} would follow.},
\begin{equation}\label{c2mupiu}\begin{split}&
\int_{(0, 1)} \frac{c_{N, s}}{|x-y|^{N+2s}} \,d\mu^+(s) 
\ge\int_{[\overline s, 1-\delta]} \frac{c_{N, s}}{|x-y|^{N+2s}} \,d\mu^+(s)
\ge\frac{\underline{c}_N\,\delta}{|x-y|^N}
\int_{[\overline s, 1-\delta]} \frac{d\mu^+(s)}{|x-y|^{2s}}
.\end{split}
\end{equation}
In virtue of~\eqref{c1mumeno} and Lemma~\ref{c1mumeno.le}, used here with~$S:=1-\delta\in[\overline s,1)$, we conclude that
$$ \int_{(0, \overline s)} \frac{c_{N, s}}{|x-y|^{N+2s}} \,d\mu^-(s) \le \frac{\overline{c}_N
\,\max\{(2R)^{2\overline{s}},(2R)^2\}\,
\mu^-\big( (0,\overline{s})\big)
}{\min\{1,(2R)^{2\overline{s}}\}\,\mu^+\big( [\overline{s},1-\delta]\big)\,|x-y|^N} \int_{[ \overline s,1-\delta]}\frac{d\mu^+(s)}{|x-y|^{2 s}}.$$
Comparing this with~\eqref{c2mupiu} and using~\eqref{mu2forte} 
and the bounds~\eqref{UPPERBCNS} and~\eqref{LOWBCNS} in Lemma~\ref{lemmacns}, we find that
\[
\begin{split}
\int_{(0, \overline s)} \frac{c_{N, s}}{|x-y|^{N+2s}} \,d\mu^-(s) &\le \frac{\overline{c}_N\,\max\{(2R)^{2\overline{s}},(2R)^2\}\,\mu^-\big( (0,\overline{s})\big)
}{\underline{c}_N\,\delta\,\min\{1,(2R)^{2\overline{s}}\}\,\mu^+\big( [\overline{s},1-\delta]\big)} \int_{(0, 1)} \frac{c_{N, s}}{|x-y|^{N+2s}} \,d\mu^+(s)\\
&< \frac{\overline{c}_N\,\underline{\Gamma}_N\overline{s}\,\max\{(2R)^{2\overline{s}},(2R)^2\}
}{4\underline{c}_N\,\overline{\Gamma}_N\,\min\{1,(2R)^{2\overline{s}}\} } \int_{(0, 1)} \frac{c_{N, s}}{|x-y|^{N+2s}} \,d\mu^+(s)
\\
& =\frac{\max\{(2R)^{2\overline{s}},(2R)^2\}}{\min\{1,(2R)^{2\overline{s}}\}\, \max\{1, (2R)^2\}} \int_{(0, 1)} \frac{c_{N, s}}{|x-y|^{N+2s}} \,d\mu^+(s)\\
&=\int_{(0, 1)} \frac{c_{N, s}}{|x-y|^{N+2s}} \,d\mu^+(s).
\end{split}
\]
This yields~\eqref{crucial.eq.di}, as desired.

We now check the last statement in Lemma~\ref{lemmamodulo}.
For this, we observe that,
in light of~\eqref{crucial.eq.diPRE} and~\eqref{crucial.eq.di},
equality in~\eqref{crucial.eq} holds true if, for all~$x$, $y\in\Omega$,
$$ u^+(x) u^-(y)+ u^-(x)u^+(y)=0.$$
Suppose also that~$u$ is not identically zero, otherwise the claim is trivial.

Without loss of generality,
we can suppose that~$u>0$ in a subset~$U$ of~$\Omega$ with positive measure.
Then, for all~$x\in U$ and~$y\in\Omega$, we have that~$ u(x) u^-(y)=0$.
Therefore, $u^-(y)=0$ for all~$y\in\Omega$, and thus~$u\ge0$ in~$\Omega$. This completes the proof of Lemma~\ref{lemmamodulo}.
\end{proof}

We point out that the assumption~\eqref{mu2forte} is essential for Lemma~\ref{lemmamodulo} to hold. On this matter, we state the following:

\begin{proposition}\label{propositionfails}
There exist a measure~$\mu$ satisfying~\eqref{mu0}, \eqref{mu1} and~\eqref{mu2}, but not~\eqref{mu2forte}, and a function~$u\in X(\Omega)$ for which $$
\int_{[0, 1]} [|u|]^2_s \, d\mu(s) > \int_{[0, 1]} [u]^2_s \, d\mu(s).$$

In particular, the statement of Lemma~\ref{lemmamodulo} does not hold in this case.
\end{proposition}

\begin{proof}
Let~$\delta_1$ and~$\delta_s$ be the Dirac measures at the points~$1$ and~$s\in(0, 1)$, respectively.  
We take~$\mu^+:=\delta_1$ and~$\mu^- := \alpha\delta_s$ for some~$\alpha\in (0, \frac{1}{2c_0})$, being~$c_0>0$ as in Proposition~\ref{crucial}. Hence, in this case we have that~$
\mu=\delta_1-\alpha\delta_s$. 

Thus, \eqref{mu0} and~\eqref{mu1} are satisfied with~$\overline s:=\frac{s+1}{2}$. Moreover, 
\[
\mu^-([0, \overline s))=
\alpha \le \frac{1}{2c_0} = \frac{1}{2c_0} \mu^+([\overline s, 1]), 
\]
that is~\eqref{mu2} is satisfied with~$\gamma:=\frac{1}{2c_0}$.

On the other hand, for any~$\delta\in(0,1-\overline{s}]$, we have that
\[
\mu^-([0, \overline s)) =
\alpha > 0 = \mu^+([\overline s, 1-\delta])
\]
that is, \eqref{mu2forte} is not satisfied.

Under these assumptions, we see that
\begin{equation}\label{a0scjo39rtgjv9H9WDFIVN9dck}
\begin{split}
\int_{[0, 1]} [|u|]^2_s \, d\mu(s) - \int_{[0, 1]} [u]^2_s \, d\mu(s)&= \int_\Omega |\nabla |u||^2 \,dx -  \int_\Omega |\nabla u|^2 \,dx - \alpha [|u|]^2_s +\alpha [u]^2_s\\
&= -\alpha\big([|u|]^2_ s - [u]^2_s\big).
\end{split}
\end{equation}
Now, choosing~$u$ as the sum of a positive and a negative bump in~$\Omega$,
exploiting~\eqref{feeeyt9tirhfcs2345678hgbhfndjksjri4yti} we see that
\begin{eqnarray*}
&& \frac14\big([|u|]^2_ s - [u]^2_s\big)
=-\iint_{\Omega\times\Omega}\frac{ u^+(x)u^-(y)+u^-(x)u^+(y)}{|x-y|^{N+2s}}\,dx\,dy=
-2\iint_{\Omega\times\Omega}\frac{ u^+(x)u^-(y)}{|x-y|^{N+2s}}\,dx\,dy<0.
\end{eqnarray*}
This and~\eqref{a0scjo39rtgjv9H9WDFIVN9dck} give the desired result.
\end{proof}

Next, we prove that if a sequence is bounded in~$X(\Omega)$,
then it ``converges strongly" with respect to the negative part of 
the measure~$\mu$.

\begin{lemma}\label{lemmalimitemumeno}
Let~$\mu$ satisfy~\eqref{mu0} and~\eqref{mu1} and let~$s_\sharp$ be as in~\eqref{scritico}. Let~$u_n$ be a sequence in~$X(\Omega)$ such that~$u_n$ converges weakly to some~$u$ in~$X(\Omega)$
as~$n\to+\infty$.

Then,
\begin{equation}\label{limitemumeno1}
\lim_{n\to +\infty}  \int_{[0,\overline s)} [u_n]^2_s \, d\mu^-(s) =  \int_{[0,\overline s)} [u]^2_s \, d\mu^-(s).
\end{equation}
\end{lemma}

\begin{proof}
We observe that, since~$u_n$ converges weakly in~$X(\Omega)$, by Proposition~\ref{embeddingsXOmega} we deduce that
\begin{equation}\label{convL2}
{\mbox{$u_n\to u$ in~$L^2(\Omega)$ as~$n\to +\infty$.}}
\end{equation}

Also, we claim that
\begin{equation}\label{claimnormafinita}
 \|u_n - u\|_{H^{\overline s}(\R^N)} < +\infty.
\end{equation}
To prove this, it enough to show that~$[u_n- u]_{\overline s}<+\infty$. For this, we notice that,
by Lemma~\ref{nons} and~\eqref{mu0},
\begin{eqnarray*}&&
+\infty > \int_{[0, 1]} [u_n - u]^2_s \, d\mu^+(s) \ge \int_{[\overline s, 1]} [u_n - u]^2_s \, d\mu^+(s)\\&&\qquad
\ge \frac{1}{C^2}  \int_{[\overline s, 1]} [u_n - u]^2_{\overline s} \, d\mu^+(s)
=  \frac{\mu^+([\overline s, 1])}{C^2} \, [u_n - u]^2_{\overline s} ,
\end{eqnarray*}
which proves the claim in~\eqref{claimnormafinita}.

Now, for any~$s\in [0,\overline s)$, we can exploit~\cite[Theorem~1]{MR3813967} and obtain that
for any~$\theta\in (0, 1]$ there exists a positive constant~$C = C(N, \theta, \overline s)$ such that
\[
\|u_n - u\|_{H^s(\R^N)}\le C \|u_n - u\|^{\theta}_{L^2(\R^N)} \|u_n - u\|^{1-\theta}_{H^{\overline s}(\R^N)}.
\]
This, \eqref{convL2} and~\eqref{claimnormafinita} entail that
\[
\begin{split}
\lim_{n\to +\infty}\int_{[0, \overline s)} [u_n - u]^2_s \,d\mu^-(s) &\le \lim_{n\to +\infty}\int_{[0, \overline s)} \|u_n - u\|^2_{H^s(\R^N)} \,d\mu^-(s)\\
&\le C^2 \lim_{n\to +\infty}\int_{[0, \overline s)} \|u_n - u\|^{2\theta}_{L^2(\R^N)} \|u_n - u\|^{2(1-\theta)}_{H^{\overline s}(\R^N)} \,d\mu^-(s)\\
&= C^2 \, \mu^-([0, \overline s)) \lim_{n\to +\infty} \|u_n - u\|^{2\theta}_{L^2(\R^N)} \|u_n - u\|^{2(1-\theta)}_{H^{\overline s}(\R^N)}\\
&=0.
\end{split}
\]
This, together with~\cite[Lemma~5.9]{DPSV2}, gives the desired result.
\end{proof}

\section{The eigenvalue problem driven by the operator~$L_\mu$}\label{sectioneigenvalues}

In this section we address the study of the eigenvalue problem driven by the operator~$L_\mu$.

We stress that the first eigenvalue of the eigenvalue problem in~\eqref{eigenvalueproblem} 
will play a crucial role in the study of stationary solutions of the logistic  diffusion problem~\eqref{problemalogistico}. 

We start by providing some preliminary notations. 
We recall that the space~$X(\Omega)$ has been defined in~\eqref{XOmegadefn}. 

The weak formulation of the eigenvalue problem~\eqref{eigenvalueproblem} is given by
\begin{equation}\label{weakeigenvalue}
\begin{split}
&\int_{[0,1]}c_{N,s}\iint_{\R^{2N}}\frac{(u(x)-u(y))(v(x)-v(y))}{|x-y|^{N+2s}}\,dx\,dy\,d\mu^+(s)\\
&\quad -  \int_{[0,\overline s)}c_{N,s}\iint_{\R^{2N}}\frac{(u(x)-u(y))(v(x)-v(y))}{|x-y|^{N+2s}}\,dx\,dy\,d\mu^-(s) \\
&= \lambda\int_\Omega u(x) v(x)\, dx \qquad\mbox{ for any } v\in X(\Omega),
\end{split}
\end{equation}
where the notation in footnote~\ref{footimpr2} is assumed.

We recall that if there exists a nontrivial solution~$u\in X(\Omega)$ of~\eqref{weakeigenvalue}, then~$\lambda\in\R$ is called an eigenvalue of the operator~$L_\mu$. Any solution~$u\in X(\Omega)$ is called an eigenfunction associated with the eigenvalue~$\lambda$.

Furthermore, let~$I:X(\Omega)\to\R$ be the functional defined as
\begin{equation}\label{Ifunctional}
\begin{split}
I(u)&:= \frac12 \int_{[0, 1]} [u]^2_s\, d\mu^+(s) - \frac12\int_{[0, \overline s)} [u]^2_s\, d\mu^-(s) \\&= \frac12 \|u\|^2_X - \frac12\int_{[0, \overline s)} [u]^2_s\, d\mu^-(s),
\end{split}
\end{equation}
where~$\|u\|_{X}$ is the seminorm given in~\eqref{normadefinizione}.

In this setting, we have the following:

\begin{lemma}\label{lemmino1}
Let~$\mu$ satisfy~\eqref{mu0}, \eqref{mu1} and~\eqref{mu2}.  
Let~$s_\sharp$ be as in~\eqref{scritico}.

Let~$X_0$ be a nonempty, weakly closed subspace of~$X(\Omega)$ and
\[
\mathscr M:=\big\{ u\in X_0\;{\mbox{ s.t. }}\; \|u\|_{L^2(\Omega)}=1\big\}.
\]

Then, there exists~$\gamma_0>0$, depending only on~$N$ and~$\Omega$,
such that if~$\gamma\in[0,\gamma_0]$ the following statements hold true.

There exists~$u_0\in\mathscr M$ such that
\begin{equation}\label{minimo1}
\min_{u\in\mathscr M} I(u) = I(u_0)>0.
\end{equation}

In addition, for any~$v\in X_0$,
\begin{equation}\label{minimo2}
\begin{split}
&\int_{[0,1]}c_{N,s}\iint_{\R^{2N}}\frac{(u_0(x)-u_0(y))(v(x)-v(y))}{|x-y|^{N+2s}}\,dx\,dy\,d\mu^+(s)\\
&\quad -  \int_{[0,\overline s)}c_{N,s}\iint_{\R^{2N}}\frac{(u_0(x)-u_0(y))(v(x)-v(y))}{|x-y|^{N+2s}}\,dx\,dy\,d\mu^-(s) = 2I(u_0)\int_\Omega u_0(x) v(x)\, dx.
\end{split}
\end{equation}
\end{lemma}

\begin{proof}
We take a minimizing sequence~$u_n\in\mathscr M$ for the functional~$I$, that is 
\begin{equation}\label{uinjomk}
\lim_{n\to +\infty} I(u_n) = \inf_{u\in\mathscr M} I(u).
\end{equation}
Accordingly,
\begin{equation}\label{ureidowvbjcask}
{\mbox{$I(u_n)\le C$ for some~$C>0$ independent of~$n$.}}\end{equation}

We point out that, in light of~\eqref{Ifunctional} and Proposition~\ref{crucial}, one has that
\begin{equation*}
I(u) \ge \frac12 (1-c_0\gamma)\|u\|^2_{X}.
\end{equation*}
Here above~$\gamma$ is the quantity appearing in the assumption~\eqref{mu2}.
Thus, if~$\gamma$ is small enough, possibly in dependence of~$c_0$, and therefore
of~$N$ and~$\Omega$, we find that, for all~$u\in X(\Omega)$,
\begin{equation}\label{574839dgmzsbc87654zsbmf4yut5we}
I(u)\ge c \|u\|^2_{X},\end{equation}
for some~$c>0$.

{F}rom this fact and~\eqref{ureidowvbjcask}, we infer that
\begin{equation}\label{unisboundedinX}
u_n \, \mbox{ is bounded in~$X(\Omega)$}.
\end{equation}
Hence, up to subsequences, there exists~$u_0\in X_0$ such that~$u_n$ converges weakly in~$X(\Omega)$ and strongly in~$L^2(\R^N)$  to~$u_0$
(thanks to Proposition~\ref{embeddingsXOmega}).
Moreover, we have that~$\|u_0\|_{L^2(\Omega)}=1$, and therefore~$u_0\in\mathscr M$.

We observe that, by Fatou's lemma, 
\begin{equation}\label{uyhbijnklm}
\begin{split}
&\liminf_{n\to +\infty} \int_{[0,1]}c_{N,s}\iint_{\R^{2N}}\frac{|u_n(x)-u_n(y)|^2}{|x-y|^{N+2s}}\,dx\,dy\,d\mu^+(s)\\
&\qquad\ge \int_{[0,1]}c_{N,s}\iint_{\R^{2N}}\frac{|u_0(x)-u_0(y)|^2}{|x-y|^{N+2s}}\,dx\,dy\,d\mu^+(s).
\end{split}
\end{equation}
{F}rom this and Lemma~\ref{lemmalimitemumeno}, we conclude that
\[
\begin{split}
\lim_{n\to +\infty} I(u_n) &=\frac12 \lim_{n\to +\infty} \left(\,\int_{[0,1]} [u_n]^2_s \, d\mu^+(s) - \int_{[0,\overline s)} [u_n]^2_s \, d\mu^-(s)\right)\\
&\ge \int_{[0,1]} [u_0]^2_s \, d\mu^+(s) - \int_{[0,\overline s)} [u_0]^2_s \, d\mu^-(s)\\
&=I(u_0) \\&\ge \inf_{u\in\mathscr M} I(u).
\end{split}
\]
Combining this and~\eqref{uinjomk}, we get the existence of a minimizer in~\eqref{minimo1}.

Moreover, we stress that~$I(u_0)>0$. Indeed,
since~$u_0\in\mathscr M$, we have that~$u_0 \not\equiv 0$. Hence, by~\eqref{Ifunctional} and~\eqref{574839dgmzsbc87654zsbmf4yut5we},
\[
I(u_0)\ge c \|u_0\|^2_X >0,
\]
as desired.

We now focus on the proof of~\eqref{minimo2}. To this end, let~$\varepsilon\in (-1, 1)\setminus\{0\}$ and~$v\in X_0$ and set
\[
u_\varepsilon(x):= \dfrac{u_0(x)+\varepsilon v(x)}{\|u_0+\varepsilon v\|_{L^2(\Omega)}}.
\]
In light of this, we have that~$u_\varepsilon\in\mathscr M$. Moreover, recalling~\eqref{scalarepiu} and~\eqref{scalaremeno}, we have
\begin{align*}
& \|u_0+\varepsilon v\|^2_{L^2(\Omega)} = 1+2\varepsilon \int_\Omega u_0(x) v(x)\, dx + \varepsilon^2 \|v\|^2_{L^2(\Omega)},\\
&  \|u_0+\varepsilon v\|^2_{X} =\|u_0\|^2_X +2\varepsilon \langle u_0, v\rangle_+ +\varepsilon^2 \|v\|^2_X,\\
\mbox{ and }\quad & \int_{[0, \overline s)} [u_0+\varepsilon v]^2_s \, d\mu^-(s)= \int_{[0, \overline s)} [u_0]^2_s \, d\mu^-(s) +2\varepsilon \langle u_0, v\rangle_- +\varepsilon^2\int_{[0, \overline s)} [v]^2_s \, d\mu^-(s).
\end{align*}
{F}rom this and~\eqref{Ifunctional}, we get
\[
\begin{split}
2 I(u_\varepsilon) &= \frac{1}{\|u_0+\varepsilon v\|^2_{L^2(\Omega)}} \left(\|u_0+\varepsilon v\|^2_{X} -\int_{[0, \overline s)}[u_0+\varepsilon v]^2_s \, d\mu^-(s) \right)\\
&= \dfrac{2I(u_0) +2\varepsilon\Big(\langle u_0, v\rangle_+ - \langle u_0, v\rangle_-\Big)+\varepsilon^2 \left(\|v\|^2_X -\displaystyle\int_{[0, \overline s)} [v]^2_s \, d\mu^-(s) \right)}{1+2\varepsilon \displaystyle\int_\Omega u_0(x) v(x)\, dx + \varepsilon^2 \|v\|^2_{L^2(\Omega)}}.
\end{split}
\]
Accordingly, we have
\[
\begin{split}
\dfrac{2I(u_\varepsilon) - 2I(u_0)}{\varepsilon} =\;& \dfrac{2\left(\langle u_0, v\rangle_+ - \langle u_0, v\rangle_- -2I(u_0)\displaystyle\int_\Omega u_0(x) v(x) \,dx\right)}{1+2\varepsilon \displaystyle\int_\Omega u_0(x) v(x)\, dx + \varepsilon^2 \|v\|^2_{L^2(\Omega)}}\\
&\qquad+ \dfrac{\varepsilon \left(\|v\|^2_X -\displaystyle\int_{[0, \overline s)} [v]^2_s \, d\mu^-(s) -2I(u_0) \|v\|^2_{L^2(\Omega)}\right)}{1+2\varepsilon \displaystyle\int_\Omega u_0(x) v(x)\, dx + \varepsilon^2 \|v\|^2_{L^2(\Omega)}}.
\end{split}
\]
Since~$u_0$ is a minimizer for~$I$ in~$\mathscr M$, by taking the limit as~$\varepsilon\searrow 0$, we have the desired identity in~\eqref{minimo2}. 
This concludes the proof.
\end{proof}

We are now in the position of proving Theorem~\ref{propabc}.

\begin{proof}[Proof of the Theorem~\ref{propabc}]
The expression of~$\lambda_\mu(\Omega)$ introduced in~\eqref{primoautovalore} follows from Lemma~\ref{lemmino1} applied with~$X_0:=X(\Omega)$. 
In particular, the existence of a minimum follows from~\eqref{minimo1}. 
Moreover, thanks to~\eqref{minimo2}, we have that~$\lambda_\mu(\Omega)$ is an eigenvalue.

Also, if~$e_\mu$ is an eigenfunction corresponding to~$\lambda_\mu(\Omega)$,
testing the equation~\eqref{weakeigenvalue} for~$e_\mu$ against itself, we find that
$$
\int_{[0,1]}[e_\mu]_s^2\,d\mu^+(s) -  \int_{[0,\overline s)}[e_\mu]^2_s\,d\mu^-(s)= \lambda_\mu(\Omega)\int_\Omega e_\mu^2(x)\,dx,$$
and therefore~$e_\mu$ is a minimum for the expression in~\eqref{primoautovalore}.

We now claim that
\begin{equation}\label{tyreuiwoqcbxnmtryueiw}
{\mbox{if~\eqref{mu2forte} holds, then every eigenfunction~$e_\mu$ does not change sign.}}\end{equation}
Indeed, since for any~$s\in [0,1]$ it holds that~$[|e_\mu|]_s\le [e_\mu]_s$, we have
\[
\int_{[0, 1]} [|e_\mu|]^2_s \, d\mu^+(s) \leq\int_{[0, 1]} [e_\mu]^2_s \, d\mu^+(s).
\]
This implies that~$|e_\mu|\in X(\Omega)$.

We also point out that~$\||e_\mu|\|_{L^2(\Omega)}=\|e_\mu\|_{L^2(\Omega)}$. 

Moreover, thanks to Lemma~\ref{lemmamodulo}
(which we can exploit, since we are assuming~\eqref{mu2forte}), we know that
\begin{equation*}
\int_{[0, 1]} [|e_\mu|]^2_s \, d\mu^+(s)-\int_{[0, 1]} [|e_\mu|]^2_s \, d\mu^-(s)
\le \int_{[0, 1]} [e_\mu]^2_s \, d\mu^+(s)-\int_{[0, 1]} [e_\mu]^2_s \, d\mu^-(s)
\end{equation*}
with strict inequality if~$e_\mu$ changes sign in~$\Omega$.

These observations and the minimality of~$e_\mu$ prove the claim in~\eqref{tyreuiwoqcbxnmtryueiw}.

Thus, from now on, thanks to~\eqref{tyreuiwoqcbxnmtryueiw}, without loss of generality, we consider~$e_\mu$ a nonnegative eigenfunction.

Hence, to complete the proof of Theorem~\ref{propabc},
it remains to check that~$\lambda_\mu(\Omega)$ is simple,
provided that the assumption in~\eqref{mu2forte} is in force.

To do this, we suppose that~$f$ is an eigenfunction corresponding to the eigenvalue~$\lambda_\mu(\Omega)$ with~$\|f\|_{L^2(\Omega)}=\|e_\mu\|_{L^2(\Omega)}$ and we 
prove that
\begin{equation}\label{gfdsk7t348tyfyheqiotywoeghihwei4t637489}
{\mbox{either~$e_\mu\equiv f$ or~$e_\mu\equiv -f$.}}\end{equation}
We observe that, thanks to~\eqref{tyreuiwoqcbxnmtryueiw}, $f$ does not change sign in~$\Omega$, and we suppose, without loss of generality, that~$f\ge0$ in~$\Omega$.
With this setting, we check that~$e_\mu\equiv f$ (the other possibility being analogous).

We set~$g:=e_\mu-f$ and we claim that
\begin{equation}\label{claimguguale0}
g\equiv 0 \quad\mbox{ a.e. in } \R^N. 
\end{equation}
To prove this, we assume by contradiction that there exists a subset~$U$ of~$\Omega$
with positive measure such that~$g\ne 0$ in~$U$.

Notice that~$g$ is an eigenfunction corresponding to the eigenvalue~$\lambda_\mu(\Omega)$, and therefore, owing to~\eqref{tyreuiwoqcbxnmtryueiw}, $g$ cannot change
sign in~$\Omega$.
This yields that either~$e_\mu\geq f$ or~$e_\mu\le f$ a.e. in~$\Omega$. Since both~$e_\mu$ and~$f$ are nonnegative in~$\Omega$, we have that
\begin{equation}\label{zsedcftgbhujm}
\mbox{either}\quad e_\mu^2\geq f^2\quad\mbox{or}\quad e_\mu^2\leq f^2\quad\mbox{ a.e.  in } \Omega.
\end{equation}
Moreover,
\[
\int_\Omega\left(e_\mu^2(x)-f_1^2(x) \right)\,dx=\|e_\mu\|_{L^2(\Omega)}^2-\|f_1\|_{L^2(\Omega)}^2=0.
\]
This, together with~\eqref{zsedcftgbhujm}, gives that~$e_\mu^2-f^2=0$ a.e. in~$\Omega$, which implies that~$g=0$ a.e. in~$\Omega$, and therefore in~$\R^N$.
We have thus reached the desired contradiction, which establishes the claim in~\eqref{claimguguale0}, and so the one in~\eqref{gfdsk7t348tyfyheqiotywoeghihwei4t637489}, as desired.
\end{proof}

Since it will be useful in the study of the logistic diffusion problem~\eqref{problemalogistico}, we provide here the following scaling property for the first eigenvalue:

\begin{proposition}\label{scalingprimoautovalore}
Let~$\Omega$ be an open and bounded subset of~$\R^N$
with Lipschitz boundary. Let~$\mu$ satisfy~\eqref{mu0} and set
$$ \Sigma:=\mbox{supp}(\mu^+).$$

Let also~$r>0$ and
\[
\Omega_r:=\{ rx \;{\mbox{ with }}\; x\in\Omega\}.
\]

Then,
\[
\inf_{s\in\Sigma} (r^{2s})\, \lambda_{\mu^+}(\Omega_r)\le \lambda_{\mu^+}(\Omega)\le \sup_{s\in\Sigma} (r^{2s})\, \lambda_{\mu^+}(\Omega_r).
\]
\end{proposition}

\begin{proof}
Let
\begin{equation}\label{primoautovalore+}
\lambda_{\mu^+}(\Omega):=\min_{u\in X(\Omega)\setminus\{0\}} \dfrac{\displaystyle\int_{[0, 1]} [u]^2_s \, d\mu^+(s)}{\displaystyle\int_\Omega |u(x)|^2 \,dx}.
\end{equation}
Given~$u\in X(\Omega)$, we set
\[
\overline u(x):=u\left(\frac{x}{r}\right).
\]
By definition, $\overline u\in X(\Omega_r)$. Moreover, we observe that
\[
[\overline u]_s^2 = r^{N-2s} [u]^2_s, \quad\mbox{ and }\quad \|\overline u\|^2_{L^2(\Omega_r)} = r^N \|u\|^2_{L^2(\Omega)}.
\]
{F}rom this and~\eqref{primoautovalore+}, we infer that
\begin{equation}\label{ughbinjokl1}
\begin{split}&
\lambda_{\mu^+}(\Omega) = \min_{u\in X(\Omega)\setminus\{0\}} \dfrac{\displaystyle\int_{[0, 1]} [u]^2_s \, d\mu^+(s)}{\|u\|^2_{L^2(\Omega)}} = \min_{\overline u\in X(\Omega_r)\setminus\{0\}} \dfrac{\displaystyle\int_{[0, 1]} r^{-N+2s} [\overline u]^2_s \, d\mu^+(s)}{r^{-N}\|\overline u\|^2_{L^2(\Omega_r)}}\\
&\qquad\le \sup_{s\in\Sigma} (r^{2s}) \min_{\overline u\in X(\Omega_r)\setminus\{0\}} \dfrac{\displaystyle\int_{[0, 1]} [\overline u]^2_s \, d\mu^+(s)}{\|\overline u\|^2_{L^2(\Omega_r)}} = \sup_{s\in\Sigma} (r^{2s})\,\lambda_{\mu^+}(\Omega_r).
\end{split}
\end{equation}

Similarly, one has that
\begin{equation}\label{ughbinjokl}
\lambda_{\mu^+}(\Omega)\ge\inf_{s\in\Sigma} (r^{2s}) \min_{\overline u\in X(\Omega_r)\setminus\{0\}} \dfrac{\displaystyle\int_{[0, 1]} [\overline u]^2_s \, d\mu^+(s)}{\|\overline u\|^2_{L^2(\Omega_r)}} = \inf_{s\in\Sigma} (r^{2s})\,\lambda_{\mu^+}(\Omega_r).
\end{equation}
Combining~\eqref{ughbinjokl1} and~\eqref{ughbinjokl} we obtain the desired result.
\end{proof}

\section{The logistic diffusion equation governed by the operator~$L_\mu$}\label{sectionlogistica}
In this section we address the study of problem~\eqref{problemalogistico} and we provide the proofs of Theorems~\ref{thmlogistica1}, \ref{thmlogistica2}, \ref{thmlogistica3}, \ref{thmlogistica4} and~\ref{thmduemisure} and Proposition~\ref{proplogistica5}.
To start with, we introduce the following definition (recall that the notation in footnote~\ref{footimpr2} is assumed):

\begin{definition}\label{defweaksolution}
We say that~$u\in X(\Omega)$, with~$u\ge 0$, is a weak solution of problem~\eqref{problemalogistico} if, for any~$v\in C_c^\infty (\Omega)$,
\[
\begin{split}
&\int_{[0, 1]} c_{N, s}\iint_{\R^{2N}} \frac{(u(x)-u(y))(v(x)-v(y))}{|x-y|^{N+2s}} \,dx\, dy\, d\mu^+(s)\\
&\quad-\int_{[0, \overline s)} c_{N, s}\iint_{\R^{2N}} \frac{(u(x)-u(y))(v(x)-v(y))}{|x-y|^{N+2s}} \,dx\, dy\, d\mu^-(s)\\
&\quad =\int_\Omega \sigma(x) u(x) v(x)\,dx -
\int_\Omega \nu(x) u^2(x) v(x)\, dx +\tau\int_\Omega (J\ast u)(x) v(x)\, dx.
\end{split}
\]
\end{definition}

We give the proof of Theorem~\ref{thmlogistica1} using a direct minimization argument. To this aim, we recall a useful property stated in~\cite[Lemma~2.3]{MR3579567}:

\begin{lemma}\label{lemmaJ}
Let~$v$, $w\in L^2(\Omega)$ be such that~$v(x)=w(x)=0$ for a.e.~$x\in \R^N\setminus \Omega$. 

Then,
\[
\int_\Omega v(x)(J*w)(x)\,dx\le \|v\|_{L^2(\Omega)} \|w\|_{L^2(\Omega)}.
\]
\end{lemma}

Let~$E:X(\Omega)\to \R$ be the functional defined as 
\begin{equation}\label{defE}
E(u):=\frac12\int_{[0,1]}[u]_s^2\,d\mu^+(s) -\frac12\int_{[0,\overline s)}[u]_s^2\,d\mu^-(s)+\int_{\Omega}\frac{\nu|u|^3}{3}\,dx-\int_{\Omega}\frac{\sigma u^2}{2}\,dx -\int_{\Omega}\frac{\tau (J*u) u}{2}\,dx.
\end{equation}

We point out that the functional in~\eqref{defE} might be unbounded from above.  However, in our analysis,  this is not relevant as we find a solution which is an absolute minimum for the functional~$E$ in~\eqref{defE}.

For the sake of completeness, we show that, under the additional assumption
\begin{equation}\label{ipotesinuinLq}
\nu \in L^q(\Omega) \ \mbox{ for } q=\frac{2^*_{s_\sharp}}{2^*_{s_\sharp}-3}\in (1, +\infty], 
\end{equation}
one has that~$|E(u)|<+\infty$. Also, we observe that hypothesis~\eqref{ipotesinuinLq} leads to a constraint on the dimension~$N$, that is~$N\le 6s_\sharp$.
This is one of the main reasons for us to work in a general setting that does not require assumption~\eqref{ipotesinuinLq}: this causes the additional difficulty of having to handle a possible unbounded functional, but it provides the benefit of obtaining our main results without any dimensional restriction.

\begin{lemma}
Let~$\mu$ satisfy~\eqref{mu0}, \eqref{mu1} and~\eqref{mu2}.
Let~$\tau\ge 0$ and assume that~\eqref{ipotesilogistica1} hold true.

Then, for every~$u\in X(\Omega)$, we have that
$$ |E(u)|<+\infty.$$
\end{lemma}

\begin{proof}
Let~$u\in X(\Omega)$. Hence, by Proposition~\ref{crucial},
\[
\left|\,\int_{[{{ 0 }}, \overline s)} [u]_{s}^2 \, d\mu^- (s)\right| \le c_0\,\gamma \int_{[\overline s, 1]} [u]^2_{s} \, d\mu^+(s) \le c_0\,\gamma \int_{[0, 1]} [u]^2_{s} \, d\mu^+(s) = c_0 \gamma \|u\|^2_X.
\]
Moreover, by~\eqref{ipotesilogistica1}, \eqref{ipotesinuinLq}, the H\"older inequality and Proposition~\ref{embeddingsXOmega}, we have that
\[
\begin{split}&
\left|\int_{\Omega}\nu|u|^3 \,dx\right|
\le \left(\int_\Omega\nu^{\frac{2^*_{s_\sharp}}{2^*_{s_{\sharp}}-3}} \,dx \right)^{\frac{2^*_{s_{\sharp}}-3}{2^*_{s_\sharp}}} \left(\int_\Omega |u|^{2^*_{s_\sharp}} \,dx \right)^{\frac{3}{2^*_{s_\sharp}}} = \|\nu\|_{L^{2^*_{s_\sharp}/(2^*_{s_{\sharp}}-3)}(\Omega)} \| u\|^3_{L^{2^*_{s_\sharp}}(\Omega)}\\
&\qquad\qquad\le C_1 \|\nu\|_{L^{2^*_{s_\sharp}/(2^*_{s_{\sharp}}-3)}(\Omega)} \| u\|^3_X
\end{split}
\]
and 
\[
\begin{split}
&\left|\int_{\Omega}\sigma u^2 \,dx\right|\le \left(\int_\Omega \sigma^{\frac{2^*_{s_\sharp}}{2^*_{s_\sharp}-2}} \right)^{\frac{2^*_{s_\sharp}-2}{2^*_{s_\sharp}}} 
\left(\int_\Omega u^{2^*_{s_\sharp}}\right)^{\frac{2}{2^*_{s_\sharp}}} = \|\sigma\|_{L^{2^*_{s_\sharp}/(2^*_{s_{\sharp}}-2)}(\Omega)} \| u\|^2_{L^{2^*_{s_\sharp}}(\Omega)}\\
&\qquad\qquad\le C_2 \|\sigma\|_{L^{2^*_{s_\sharp}/(2^*_{s_{\sharp}}-2)}(\Omega)} \| u\|^2_X,
\end{split}
\]
for some constants~$C_1$, $C_2>0$.

In addition, by using Lemma~\ref{lemmaJ} with~$v:=u$ and~$w:=u$ and Proposition~\ref{embeddingsXOmega}, we we infer that there exists~$C_3>0$ such that
\[
\left|\int_{\Omega}\frac{\tau u(J\ast u)}{2}\,dx\right|\le \frac{\tau}{2} \|u\|^2_{L^2(\Omega)}\le \frac{C_3 \tau}{2} \|u\|^2_X.
\]
{F}rom these considerations, we deduce that~$|E(u)|<+\infty$, as desired.
\end{proof}

\begin{proposition}\label{proppunticritici}
Let~$\mu$ satisfy~\eqref{mu0}, \eqref{mu1} and~\eqref{mu2}.
Let~$\tau\ge 0$ and assume that~\eqref{ipotesilogistica1} holds true.

Then,
the nonnegative critical points of~$E$ are nonnegative weak solutions of problem~\eqref{problemalogistico}.
\end{proposition}

\begin{proof}
We evaluate the first variation of the functional in~\eqref{defE}. We start by computing the first variation of the functional
\[
J(u):= \frac12\int_{[0,1]}[u]_s^2\,d\mu^+(s) -\frac12\int_{[0,\overline s)}[u]_s^2\,d\mu^-(s)+\int_{\Omega}\frac{\nu|u|^3}{3}\,dx-\int_{\Omega}\frac{\sigma u^2}{2}\,dx.
\]
To do this, we take~$\varepsilon\in (-1, 1)\setminus\{0\}$ and~$v\in C_c^\infty(\Omega)$ and, recalling also the notation in~\eqref{scalarepiu} and~\eqref{scalaremeno}, we see that
\[
\begin{split}&
J(u+\varepsilon v) \\&= \frac12\int_{[0, 1]}[u+\varepsilon v]^2_s\, d\mu^+(s)-
\frac12\int_{[0, \overline s)} [u+\varepsilon v]^2_s\, d\mu^-(s)\\
&\quad+\frac{1}{3} \int_{\Omega}\nu\,\mbox{sign}(u+\varepsilon v)(u+\varepsilon v)^3 \,dx-\frac{1}{2}\int_{\Omega}\sigma\,(u+\varepsilon v)^2 \,dx\\
&= J(u) -\int_\Omega \frac{\nu |u|^3}{3} \,dx +\int_{\Omega}\frac{\nu}{3}\,\mbox{sign}(u+\varepsilon v)\, u^3 \, dx\\
&\quad+\varepsilon \left(\langle u, v\rangle_+ -\langle u, v\rangle_- +
\int_\Omega \nu\, \mbox{sign}(u+\varepsilon v)\, u^2 v \, dx -\int_\Omega \sigma u v \,dx\right)\\
&\quad+ \varepsilon^2 \left(\frac12\int_{[0, 1]} [v]^2_s\, d\mu^+(s) - \frac12
\int_{[0, \overline s)}[v]^2_s\, d\mu^-(s)+
\int_\Omega \nu\, \mbox{sign}(u+\varepsilon v)\, u\,v^2 \, dx - \int_\Omega \frac{\sigma\, v^2}{2} \,dx \right)\\
&\quad+\frac{\varepsilon^3}{3}\int_\Omega \nu\, \mbox{sign}(u+\varepsilon v)\, v^3\, dx.
\end{split}
\]
As a result, we get that
\begin{equation}\label{84932eyfnkdyr2387542iedgvchbjswjk}
\lim_{\varepsilon\to 0} \frac{J(u+\varepsilon v) - J(u)}{\varepsilon} = \langle u, v\rangle_+ -\langle u, v\rangle_- + \int_\Omega \nu\, |u|u v\, dx -\int_\Omega \sigma\, u v\,dx.
\end{equation}

Now we point out that
$$E(u)= J(u)-\int_{\Omega}\frac{\tau (J*u) u}{2}\,dx.$$
Therefore, from~\eqref{84932eyfnkdyr2387542iedgvchbjswjk}
and the second-last formula in the proof of Lemma~2.2 in~\cite{MR3579567}, we infer that, for any~$v\in C_c^\infty(\Omega)$,
\[
\begin{split}
\frac{d}{d\varepsilon}E(u+\varepsilon v)\big|_{\varepsilon=0} &= \int_{[0, 1]} c_{N, s}\iint_{\R^{2N}} \frac{(u(x)-u(y))(v(x)-v(y))}{|x-y|^{N+2s}} \,dx\, dy\, d\mu^+(s)\\
&\quad - \int_{[0, \overline s)} c_{N, s}\iint_{\R^{2N}} \frac{(u(x)-u(y))(v(x)-v(y))}{|x-y|^{N+2s}} \,dx\, dy\, d\mu^-(s)\\
&\quad+ \int_\Omega \nu \,|u|u v \, dx -\int_\Omega \sigma\, u v\, dx -\int_\Omega \tau (J\ast u)(x)v(x) \, dx.
\end{split}
\]
Hence, if~$u$ is a nonnegative critical point of the functional~\eqref{defE},
then~$u$ is a nonnegative weak solution of~\eqref{problemalogistico}, according to the
Definition~\ref{defweaksolution}.
\end{proof}

With this, we can now establish the following existence result:

\begin{proposition}\label{propesistenza}
Let~$\Omega$ be an open and bounded subset of~$\R^N$ with Lipschitz boundary.

Then, there exists~$\gamma_0>0$, depending only on~$N$ and~$\Omega$,
such that if~$\gamma\in[0,\gamma_0]$ the following statements hold true.

Let~$\mu$ satisfy~\eqref{mu0}, \eqref{mu1} and~\eqref{mu2}.
Let~$s_\sharp$ be as in~\eqref{scritico}.
Let~$\sigma$ and~$\nu$ satisfy~\eqref{ipotesilogistica1} and set
\begin{equation}\label{ipotesiminimizzante2}
p:=\frac{2m}{m-1}.
\end{equation}

Then, the functional~$E$ in~\eqref{defE} attains its minimum in~$X(\Omega)$. 

If in addition~$\mu$ satisfies~\eqref{mu2forte}, then
there exists a nonnegative minimizer~$u$ which is a solution of~\eqref{problemalogistico}.
\end{proposition}

\begin{proof}
We first observe that, from~\eqref{ipotesiminimizzante2} it follows that~$p\in [2,2^*_{s_\sharp})$ and
\begin{equation}\label{zsedcfgyhbnjio}
\frac{2}{p}+\frac{1}{m}=1.
\end{equation}
Moreover, from Lemma~\ref{lemmaJ} used with~$v:=u$ and~$w:=u$, we get that
\begin{equation}\label{xerfvgyujnjkiol}
\int_{\Omega}\frac{\tau u(J*u)}{2}\,dx\le\frac{\tau}{2}\int_{\Omega}u^2\,dx.
\end{equation}
In addition, by using the Young inequality with conjugate exponents~$3/2$ and~$3$, we obtain that
\[
\frac{(\sigma+\tau)u^2}{2}=\frac{\nu^{\frac{2}{3}}u^2}{2^\frac{2}{3}}\cdot \frac{\sigma+\tau}{2^\frac{1}{3}\nu^\frac{2}{3}}\le\frac{\nu |u|^3}{3}+\frac{(\sigma+\tau)^3}{6\nu^2}.
\]
This and~\eqref{xerfvgyujnjkiol} imply that
\[
\int_\Omega \left(\frac{\nu |u|^3}{3} -\frac{\sigma u^2}{2}-\frac{\tau u(J*u)}{2} \right) \,dx\ge -\int_\Omega\frac{(\sigma+\tau)^3}{6\nu^2}\,dx=:-\kappa.
\]
We stress that~$\kappa>0$ does not depend on~$u$ and it is finite, thanks to~\eqref{ipotesilogistica1}. 

As a result, recalling the definition of~$E$ in~\eqref{defE} and employing Proposition~\ref{crucial}, we have that, for any~$u\in X(\Omega)$,
\[
\begin{split}
E(u)&\ge \frac{1}{2}\int_{[0,1]}[u]_s^2\,d\mu^+(s)-\frac{1}{2}\int_{[0,\overline{s})}[u]_s^2\,d\mu^-(s) -\kappa
\\&\ge \frac{1-c_0 \gamma}{2}\int_{[0,1]}[u]_s^2\,d\mu^+(s)-\kappa\\
&= \frac{1-c_0 \gamma}{2} \|u\|^2_X -\kappa.
\end{split}
\]
Taking~$\gamma$ sufficiently small in dependence of~$c_0$, and therefore of~$N$ and~$\Omega$, we thus find that~$E(u)\ge-\kappa$, which
says that the functional~$E$ is bounded from below. 

Hence, we can take a minimizing sequence~$u_n$ for the functional~$E$. Notice that~$E(0)=0$. Now, if~$0$ is the desired minimizer, there we are done.
Otherwise, we can
assume that there exists~$\overline{n}\in \N$ such that,
for any~$n\geq \overline{n}$,
\[
0=E(0)\geq E(u_n)\geq \frac{1-c_0 \gamma}{2}\|u_n\|_X^2 -\kappa.
\]

This implies that 
\[
\|u_n\|_X^2\leq \frac{2\kappa}{1-c_0 \gamma},
\]
that is, the sequence~$u_n$ is bounded in~$X(\Omega)$.
Then, up to a subsequence, $u_n$ converges weakly to some~$u$ in~$X(\Omega)$. 

We claim that~$u$ is the desired minimizer.
Indeed, from Proposition~\ref{embeddingsXOmega} we get that 
\begin{equation}\label{yftguvbhijnokmpl}
u_n\to u \ \mbox{ in~$L^r(\Omega)$ for any~$r\in [2,2^*_{s_\sharp})$}
\end{equation}
and~$u_n\to u$ a.e. in~$\R^N$ as~$n\to+\infty$.
Thus, by~\eqref{zsedcfgyhbnjio}, \eqref{yftguvbhijnokmpl} and
the H\"older inequality, we get that
\begin{equation}\label{minimizzante1}
\begin{aligned}
\limsup_{n\to+\infty}&\left|\int_\Omega \sigma(u_n^2-u^2)\,dx\right|
=\limsup_{n\to+\infty}\left|\int_\Omega \sigma(u_n-u)(u_n+u)\,dx\right| \\
&\leq \limsup_{n\to+\infty}\|\sigma\|_{L^m(\Omega)}\|u_n-u\|_{L^p(\Omega)}\|u_n+u\|_{L^p(\Omega)}=0.
\end{aligned}
\end{equation}

Moreover, we can write
\begin{equation*}
\begin{aligned}
\int_\Omega \big(u_n(J*u_n)-u(J*u) \big)\,dx
&=\int_\Omega (u_n-u)(J*u_n)\,dx+\int_\Omega u\big((J*u_n)-(J*u)\big)\,dx \\
&=\int_\Omega (u_n-u)(J*u_n)\,dx+\int_\Omega u\big(J*(u_n-u)\big)\,dx
\end{aligned}
\end{equation*}
{F}rom Lemma~\ref{lemmaJ}, used with~$v:=u_n-u$ and~$w:=u_n$ and also with~$v:=u$ and~$w:=u_n-u$, and~\eqref{yftguvbhijnokmpl}, we see that
\begin{equation}\label{zawedxcftgbhuj1}
\begin{split}&\limsup_{n\to+\infty}\left|\int_\Omega \big(u_n(J*u_n)-u(J*u) \big)\,dx\right|\\
&\quad
\le \limsup_{n\to+\infty}\|u_n-u\|_{L^2(\Omega)}\left(
\|u_n\|_{L^2(\Omega)}
+\|u\|_{L^2(\Omega)}\right)
=0.
\end{split}\end{equation}

Furthermore, since~$u_n$ converges weakly to~$u$ in~$X(\Omega)$, we can employ Lemma~\ref{lemmalimitemumeno} to find that
\begin{equation}\label{minimizzante2.5}
\lim_{n\to+\infty}\int_{[0,\overline s)}[u_n]_s^2\,d\mu^-(s)= \int_{[0,\overline s)}[u]_s^2\,d\mu^-(s)
\end{equation}
In addition, by Fatou's Lemma,
\begin{equation}\label{minimizzante3}
\liminf_{n\to+\infty}\int_{[0,1]}[u_n]_s^2\,d\mu^+(s)\ge \int_{[0,1]}[u]_s^2\,d\mu^+(s)
\end{equation}
and
\begin{equation}\label{minimizzante4}
\liminf_{n\to+\infty} \int_\Omega \nu |u_n|^3 \,dx\geq \int_\Omega \nu |u|^3 \,dx.
\end{equation}

Thus, by~\eqref{defE}, \eqref{minimizzante1}, \eqref{zawedxcftgbhuj1}, \eqref{minimizzante2.5}, \eqref{minimizzante3} and~\eqref{minimizzante4} we deduce that
\[
\liminf_{n\to+\infty} E(u_n)\ge E(u),
\]
and therefore~$u$ is the desired minimizer.

If in addition~$\mu$ satisfies~\eqref{mu2forte},
we can exploit Lemma~\ref{lemmamodulo} and we have that 
\[
\int_{[0,1]}[|u|]_s^2\,d\mu^+(s)-\int_{[0,\overline{s})}[|u|]_s^2\,d\mu^-(s)\le \int_{[0,1]}[u]_s^2\,d\mu^+(s)-\int_{[0,\overline{s})}[u]_s^2\,d\mu^-(s).
\]
Hence~$E(|u|)\le E(u)$, and thus~$u$ can be assumed to be nonnegative. 

Finally, in light of Proposition~\ref{proppunticritici}, we have that~$u$ is a nonnegative weak solution of~\eqref{problemalogistico}, as desired.
\end{proof}

\subsection{Proofs of Theorems~\ref{thmlogistica1}, \ref{thmlogistica2} and~\ref{thmlogistica3}}

In this section we prove the existence results stated in the introduction for problems~\eqref{problemalogistico} and~\eqref{problemchemotaxis}.

\begin{proof}[Proof of Theorem~\ref{thmlogistica1}]
{F}rom Proposition~\ref{propesistenza} there exists a nonnegative minimizer for the functional~$E$ defined in~\eqref{defE}, which is the desired nonnegative
weak solution of problem~\eqref{problemalogistico}.
\end{proof}

\begin{proof}[Proof of Theorem~\ref{thmlogistica2}]
We first prove the first statement in Theorem~\ref{thmlogistica2}. For this, we
suppose that~\eqref{supsigma} holds true and we assume by contradiction that there exist a nonnegative and nontrivial solution of problem~\eqref{problemalogistico}.

We claim that
\begin{equation}\label{claimu>0thm2}
\int_\Omega \nu u^3 \, dx>0 .
\end{equation}
Suppose not, namely
\[
\int_\Omega \nu u^3 \, dx = 0.
\]
{F}rom the latter identity, we have that~$\nu(x) =0$ a.e.
in~$\mbox{supp}(u)$. Hence, in light of the assumption~\eqref{ipotesilogistica1}, it follows that~$\sigma(x)=0$ a.e.
in~$\mbox{supp}(u)$ and~$\tau=0$.

Using these observations and testing the equation (recall Definition~\ref{defweaksolution}) against~$u$ itself, we thus obtain that
\[
\begin{split}
0&= \int_\Omega \sigma u^2\,dx - \int_\Omega \nu u^3\, dx 
+\tau\int_\Omega  (J\ast u) u\, dx\\
&=\int_{[0, 1]} [u]^2_s\, d\mu^+(s)-\int_{[0, \overline s)} [u]^2_s\, d\mu^-(s).\end{split}
\]
Hence,
from Proposition~\ref{crucial}, and choosing~$\gamma$ sufficiently small,
in dependence of~$N$ and~$\Omega$, we conclude that
\[
0\ge(1-c_0\gamma) \int_{[0, 1]} [u]^2_s\, d\mu^+(s)\ge c\|u\|_X^2.
\]
This entails that~$u$ is constant in~$\Omega$. 

Accordingly, since~$u=0$ in~$\R^N\setminus\Omega$, we deduce that~$u\equiv 0$ in~$\R^N$, which is a contradiction, since we assumed~$u$ to be a nontrivial solution of~\eqref{problemalogistico}. This proves the claim in~\eqref{claimu>0thm2}.

Now, we make use of Lemma~\ref{lemmaJ} with~$v=w=u$, \eqref{supsigma} and~\eqref{claimu>0thm2} to get
\begin{equation*}
\begin{split}
&\int_\Omega (\sigma-\nu u) u^2 \,dx +\int_\Omega \tau(J\ast u) u\, dx \le \int_\Omega (\sigma-\nu u +\tau) u^2 \,dx\\
&\quad\le\left(\sup_\Omega \sigma +\tau\right)\|u\|^2_{L^2(\Omega)} - \int_\Omega \nu u^3 \, dx \le\lambda_\mu (\Omega) \|u\|^2_{L^2(\Omega)} - \int_\Omega \nu u^3 \, dx\\
&\quad< \lambda_\mu (\Omega) \|u\|^2_{L^2(\Omega)}.
\end{split}
\end{equation*}
As a result, by testing the equation in
Definition~\ref{defweaksolution} against~$u$ itself and recalling~\eqref{primoautovalore}, we find that
\[
\begin{split}
\lambda_\mu(\Omega)\|u\|^2_{L^2(\Omega)}&\le\int_{[0, 1]} [u]^2_s \, d\mu^+(s)-\int_{[0, \overline s)} [u]^2_s\, d\mu^-(s)\\
&=\int_\Omega (\sigma-\nu u) u^2 \,dx +\int_\Omega \tau(J\ast u) u \,dx\\
&< \lambda_\mu (\Omega) \|u\|^2_{L^2(\Omega)},
\end{split}
\]
which is a contradiction. The first statement in Theorem~\ref{thmlogistica2}
is thereby established.

We now prove the second statement in Theorem~\ref{thmlogistica2}.
Let us suppose that~\eqref{infsigma} holds true
and let~$e_\mu$ be the first eigenfunction of the operator~$L_\mu$,
as given by Theorem~\ref{propabc}.
Without loss of generality, we can assume that~$e_\mu\ge 0$ in~$\Omega$
and that~$\|e_\mu\|^2_{L^2(\Omega)}=1$.

We set
\[
C_1:= \frac12 \int_\Omega (\sigma - \lambda_\mu(\Omega)) e^2_\mu \, dx + \frac{\tau}{2}\int_{\Omega} e_\mu(J\ast e_\mu)\,dx, 
\]
being~$\lambda_\mu(\Omega)$ as in~\eqref{primoautovalore}.

Now, recalling the definition of the functional~$E$ in~\eqref{defE}, we get
\[
\begin{split}&
E(-e_\mu) \\&= \frac12\left(\,\int_{[0,1]}[e_\mu]_s^2\,d\mu^+(s) -\int_{[0,\overline s)}[e_\mu]_s^2\,d\mu^-(s)-\int_{\Omega} \sigma e_\mu^2\,dx -\int_{\Omega} \tau e_\mu(J\ast e_\mu)\,dx\right) -\frac13\int_{\Omega} \nu\, e_\mu^3 \,dx\\
&\le\frac12 \left(\int_\Omega (\lambda_\mu(\Omega) -\sigma) e^2_\mu \, dx -\int_{\Omega} \tau e_\mu(J\ast e_\mu)\,dx\right)\\
&=-C_1.
\end{split}
\]
Accordingly, we find that
\[
E(-e_\mu) <0 = E(0),
\]
and therefore~$0$ is not a minimizer for the functional~$E$. 

Thus, thanks to Proposition~\ref{propesistenza}, 
we conclude that the functional~$E$ admits a nonnegative and nontrivial minimizer, which is a solution of~\eqref{problemalogistico}, as desired.
\end{proof}

\begin{proof}[Proof of Theorem~\ref{thmlogistica3}]
We consider a nonnegative eigenfunction~$e_{\mu^+}$ for the operator~$L_{\mu^+}$,
associated with the eigenvalue~$\lambda_{\mu^+}(\Omega)$,
as defined in~\eqref{primoautovalore+}.

Then, recalling the definitions of~$\lambda_\mu(\Omega)$ and~$\lambda_{\mu^+}(\Omega)$ in~\eqref{primoautovalore} and~\eqref{primoautovalore+}, respectively, for any~$\varepsilon\in (0,1)$ we denote by~$\mu_\varepsilon:=\mu^+-\varepsilon \mu^-$ and we have that
\begin{equation}\label{t0430tyohfjfe21456vfds}
\lambda_{\mu_\varepsilon}(\Omega)\le \dfrac{\displaystyle\int_{[0, 1]} [e_{\mu^+}]^2_s \, d\mu^+(s) - \varepsilon\displaystyle\int_{[0, \overline s)} [e_{\mu^+}]^2_s \, d\mu^-(s)}{\displaystyle\int_\Omega |e_{\mu^+}(x)|^2 \,dx}
<\dfrac{\displaystyle\int_{[0, 1]} [e_{\mu^+}]^2_s \, d\mu^+(s) }{\displaystyle\int_\Omega |e_{\mu^+}(x)|^2 \,dx}
=\lambda_{\mu^+}(\Omega),
\end{equation}
where the strict inequality comes from the fact that we assume~$\mu^-\not \equiv 0$. 

Thus, we can take~$\tau \ge 0$, $\nu\in L^\infty(\Omega)$ with~$\inf_\Omega \nu>0$
and~$\sigma\in L^\infty(\Omega)$, such that
\begin{equation}\label{zsrfcvgyhbnjikm,lp}
\lambda_{\mu_\varepsilon}(\Omega)<\inf_\Omega \sigma\leq \sup_\Omega \sigma +\tau\leq \lambda_{\mu^+}(\Omega).
\end{equation}
We observe that, with this choice, the assumption in~\eqref{ipotesilogistica1} is satisfied. 

Then, by~\eqref{zsrfcvgyhbnjikm,lp} and Theorem~\ref{thmlogistica2} we infer that problem~\eqref{problemchemotaxis} only admits the trivial
solution, while problem~\eqref{problemchemotaxis2} admits a nonnegative and nontrivial solution, as desired.
\end{proof}

\subsection{Proof of Theorem~\ref{thmlogistica4}}
We now provide the proof of Theorem~\ref{thmlogistica4}.

\begin{proof}[Proof of Theorem~\ref{thmlogistica4}]
Since~$\Omega_1$ and~$\Omega_2$ are congruent, there exists a 
rigid motion~$\mathcal{R}$ such that~$\mathcal{R}(\Omega_1)=\Omega_2$.
Let~$e_1\in X(\Omega_1)$ be the eigenfunction associated with~$\lambda_\mu(\Omega_1)$ normalized in such a way that~$\|e_1\|^2_{L^2(\Omega_1)}=1$. Then, we have that~$e_1 \circ \mathcal{R}^{-1}\in X(\Omega_2)$. Also, by a suitable change of variables, we have that~$\|e_1\|_{L^2(\Omega_1)}=\|e_1\circ \mathcal{R}^{-1}\|_{L^2(\Omega_2)}=1$ and~$[e_1]_s^2=[e_1\circ \mathcal{R}^{-1}]_s^2$ for any~$s\in (0,1]$.
{F}rom these observations and recalling~\eqref{primoautovalore}, we infer that
\begin{equation}\label{congruenti}
\lambda_\mu(\Omega_1)=\lambda_\mu(\Omega_2).
\end{equation}

Now, we claim that 
\begin{equation}\label{claimunione}
\lambda_\mu(\Omega_1 \cup \Omega_2)<\lambda_\mu(\Omega_1).
\end{equation}
Indeed, let~$e_1\in X(\Omega_1)$ and~$e_2\in X(\Omega_2)$ be the 
first eigenfunctions associated with~$\lambda_\mu(\Omega_1)$ and~$\lambda_\mu(\Omega_2)$, respectively,
satisfying~$\|e_1\|_{L^2(\Omega_1)}=\|e_2\|_{L^2(\Omega_2)}=1$.
We set~$e:=e_1+e_2$. Then, $e\in X(\Omega_1\cup \Omega_2)$ and,
since~$\mbox{supp}(e_1)\cap \mbox{supp}(e_2)=\varnothing$,
\begin{equation}\label{equnione1}
\|e\|_{L^2(\Omega_1\cup \Omega_2)}^2
=\|e_1\|_{L^2(\Omega_1)}^2+\|e_2\|_{L^2(\Omega_2)}^2=2.
\end{equation}
If~$\mu^+(\{1\})\neq 0$, we also point out
that~$\mbox{supp}(\nabla e_1)\cap \mbox{supp}(\nabla e_2)=\varnothing$.

Moreover, from the fact that~$e\in X(\Omega_1\cup \Omega_2)$ and Proposition~\ref{crucial}, we infer that
\[
\int_{[0, 1]} [e]^2_s \, d\mu^+(s)<+\infty\qquad{\mbox{and}}\qquad
\int_{[0, 1]} [e]^2_s \, d\mu^-(s) <+\infty .\]
Gathering these pieces of information, we conclude that
\begin{equation}\label{equnione2}
\begin{aligned}&
\int_{[0, 1]} [e]^2_s \, d\mu(s) \\
&=\int_{[0, 1]} [e]^2_s \, d\mu^+(s) - \int_{[0, \overline s)} [e]^2_s \, d\mu^-(s)\\
&=\int_{[0, 1]} [e_1]^2_s \, d\mu^+(s) - \int_{[0, \overline s)} [e_1]^2_s \, d\mu^-(s) 
+\int_{[0, 1]} [e_2]^2_s \, d\mu^+(s) - \int_{[0, \overline s)} [e_2]^2_s \, d\mu^-(s) \\
&\quad +2\int_{(0, 1)}c_{N,s}\iint_{\R^{2N}}\frac{(e_1(x)-e_1(y))(e_2(x)-e_2(y))}{|x-y|^{N+2s}}\,dx\,dy\,d\mu^+(s)\\
&\quad +2\mu^+(\{0\})\int_{\Omega_1\cup \Omega_2}e_1(x)e_2(x)\,dx
+2\mu^+(\{1\})\int_{\Omega_1\cup \Omega_2}\nabla e_1(x)\cdot \nabla e_2(x)\,dx \\
&\quad-2\int_{(0, \overline{s})}c_{N,s}\iint_{\R^{2N}}\frac{(e_1(x)-e_1(y))(e_2(x)-e_2(y))}{|x-y|^{N+2s}}\,dx\,dy\,d\mu^-(s) \\
&\quad -2\mu^+(\{0\})\int_{\Omega_1\cup \Omega_2}e_1(x)e_2(x)\,dx \\
&=\lambda_\mu(\Omega_1)+\lambda_\mu(\Omega_2) \\
&\quad +2\int_{(0, 1)}c_{N,s}\iint_{\R^{2N}}\frac{(e_1(x)-e_1(y))(e_2(x)-e_2(y))}{|x-y|^{N+2s}}\,dx\,dy\,d\mu^+(s) \\
&\quad -2\int_{(0, \overline{s})}c_{N,s}\iint_{\R^{2N}}\frac{(e_1(x)-e_1(y))(e_2(x)-e_2(y))}{|x-y|^{N+2s}}\,dx\,dy\,d\mu^-(s).
\end{aligned}
\end{equation}

Now, since~$e_1=0$ in~$\R^N\setminus \Omega_1$ and~$e_2=0$ in~$\R^N\setminus \Omega_2$ and~$\overline{\Omega_1}\cap\overline{\Omega_2}=\varnothing$, we have that, for all~$s\in(0,1)$,
\begin{equation*}
\begin{aligned}
\iint_{\R^{2N}}&\frac{(e_1(x)-e_1(y))(e_2(x)-e_2(y))}{|x-y|^{N+2s}}\,dx\,dy  \\
&=-\iint_{\Omega_1 \times \Omega_2}\frac{e_1(x)e_2(y)}{|x-y|^{N+2s}}\,dx\,dy -\iint_{\Omega_2 \times \Omega_1}\frac{e_2(x)e_1(y)}{|x-y|^{N+2s}}\,dx\,dy.
\end{aligned}
\end{equation*}
Plugging this information into~\eqref{equnione2} we thus find that 
\begin{equation*}
\begin{aligned}
\int_{[0, 1]} &[e]^2_s \, d\mu^+(s) - \int_{[0, \overline s)}[e]^2_s \, d\mu^-(s) =\lambda_\mu(\Omega_1)+\lambda_\mu(\Omega_2) \\
&-2\iint_{\Omega_1 \times \Omega_2}e_1(x)e_2(y)\left(\;\int_{(0, 1)}\frac{c_{N,s}}{|x-y|^{N+2s}}\,d\mu^+(s)-\int_{(0, \overline{s})}\frac{c_{N,s}}{|x-y|^{N+2s}}\,d\mu^-(s)\right)\,dx\,dy \\
&-2\iint_{\Omega_2 \times \Omega_1}e_2(x)e_1(y)\left(\;\int_{(0, 1)}\frac{c_{N,s}}{|x-y|^{N+2s}}\,d\mu^+(s)-\int_{(0, \overline{s})}\frac{c_{N,s}}{|x-y|^{N+2s}}\,d\mu^-(s)\right)\,dx\,dy.
\end{aligned}
\end{equation*}
{F}rom this, recalling~\eqref{crucial.eq.di} and we conclude that
\begin{equation}\label{equnione7}
\int_{[0, 1]} [e]^2_s \, d\mu^+(s) - \int_{[0, \overline s)}[e]^2_s \, d\mu^-(s) <\lambda_\mu(\Omega_1)+\lambda_\mu(\Omega_2).
\end{equation}

As a result, the minimality of~$\lambda_\mu(\Omega_1 \cup \Omega_2)$,
\eqref{congruenti}, \eqref{equnione1} and~\eqref{equnione7} entail that
\[
\lambda_\mu(\Omega_1 \cup \Omega_2)\le \frac{\displaystyle\int_{[0, 1]} [e]^2_s \, d\mu^+(s) - \displaystyle\int_{[0, \overline s)}[e]^2_s \, d\mu^-(s)}{\displaystyle\int_{\Omega_1 \cup \Omega_2}|e(x)|^2\,dx} < \frac{\lambda_\mu(\Omega_1)+\lambda_\mu(\Omega_2)}{2}
=\lambda_\mu(\Omega_1).
\]
This proves the claim in~\eqref{claimunione}.

Hence, in light of~\eqref{claimunione},
we can take~$\tau\ge 0$, $\nu\in L^\infty(\Omega_1\cup \Omega_2)$ with~$\inf_{\Omega_1\cup \Omega_2}\nu>0$ and~$\sigma\in L^\infty(\Omega_1\cup \Omega_2)$, satisfying
\begin{equation}\label{equnione8}
\lambda_\mu(\Omega_1 \cup \Omega_2)<\inf_{\Omega_1 \cup \Omega_2}\sigma \leq \sup_{\Omega_1 \cup \Omega_2}\sigma +\tau 
\leq \lambda_\mu(\Omega_1). 
\end{equation}
We notice that~\eqref{ipotesilogistica1} is satisfied both in~$\Omega_1$ and~$\Omega_2$. 

Thus, by~\eqref{equnione8} and Theorem~\ref{thmlogistica2} we get that, for
any~$i\in \{1,2\}$, problem~\eqref{problemai} only admits the trivial
solution, while problem~\eqref{problemalogistico} (here~$\Omega:=\Omega_1\cup \Omega_2$) admits a nontrivial solution.
\end{proof}

We point out that the assumption~$\mu^+\big((0,1)\big)>0$ in Theorem~\ref{thmlogistica4} is needed to establish~\eqref{claimunione}. Indeed, if~$\mu^+\big((0,1)\big)=0$, namely if only local interactions occur, then~\eqref{claimunione} fails (see~\cite[formula~(30) in Section~4]{MR3579567}).

\subsection{Proofs of Proposition~\ref{proplogistica5} and Theorem~\ref{thmduemisure}}
We now discuss the results stated in the introduction that relate the survival of the population to the ``size'' of~$\Omega$.

\begin{proof}[Proof of Proposition~\ref{proplogistica5}]
By~\eqref{primoautovalore} and~\eqref{primoautovalore+}, we have that~$\lambda_\mu(\Omega_r)\le \lambda_{\mu^+}(\Omega_r)$, with strict inequality
if~$\mu^-\not\equiv0$ (due to~\eqref{t0430tyohfjfe21456vfds}).

{F}rom this fact, Proposition~\ref{scalingprimoautovalore} and assumption~\eqref{ipotesilambdamupiuPRE}, we infer that
\[
\lambda_{\mu}(\Omega_r) \le \lambda_{\mu^+}(\Omega_r)\le \frac{\lambda_{\mu^+}(\Omega)}{\displaystyle\inf_{s\in\Sigma} (r^{2s})}\le 1.
\]
In particular, we obtain that~$\lambda_{\mu}(\Omega_r) <1$
thanks to~\eqref{ipotesilambdamupiu}. Accordingly, the assumption in~\eqref{infsigma}
of Theorem~\ref{thmlogistica2} is satisfied with~$\sigma:=1$, and
therefore the desired result follows from Theorem~\ref{thmlogistica2}.
\end{proof}

\begin{proof}[Proof of Theorem~\ref{thmduemisure}]
We first claim that, if~$U_1\subseteq U_2 \subset \R^N$, then
\begin{equation}\label{claiminclusione}
\lambda_{\mu_i}(U_2)\leq \lambda_{\mu_i}(U_1) 
\quad \mbox{for any } i=1,2.
\end{equation}
Indeed, if~$u\in X(U_1)$, then~$u\in X(U_2)$ since~$u=0$ in~$\R^N \setminus U_1\supseteq\R^N\setminus U_2$ and
\[
(U_1\times U_1)\cup(U_1\times(\R^N\setminus U_1))\cup((\R^N\setminus U_1)\times U_1)\subseteq (U_2\times U_2)\cup(U_2\times(\R^N\setminus U_2))\cup((\R^N\setminus U_2)\times U_2).
\]
Thus, recalling the definition in~\eqref{primoautovalore+},
for any~$i=1,2$,
\[
\lambda_{\mu_i}(U_2)= \min_{u\in X(U_2)\setminus\{0\}} \dfrac{\displaystyle\int_{[0, 1]} [u]^2_s \, d\mu_i(s)}{\displaystyle\int_{U_2} |u(x)|^2 \,dx}
\leq \min_{u\in X(U_1)\setminus\{0\}} \dfrac{\displaystyle\int_{[0, 1]} [u]^2_s \, d\mu_i(s)}{\displaystyle\int_{U_1} |u(x)|^2 \,dx}
=\lambda_{\mu_i}(U_1).
\]
This establishes the claim in~\eqref{claiminclusione}.

Now, since~$\Omega$ is open and bounded, there exist~$x_0\in \Omega$ and~$R_2>R_1>0$ such that~$B_{R_1}(x_0)\subset \Omega \subset B_{R_2}(x_0)$. Thus, from~\eqref{claiminclusione} we infer that 
\begin{equation}\label{zserfcvgyuhbnjiolkm}
\lambda_{\mu_i}(B_{R_2}(x_0))\leq \lambda_{\mu_i}(\Omega)
\leq \lambda_{\mu_i}(B_{R_1}(x_0)) \quad \mbox{for any } i=1,2.
\end{equation}
Moreover, Theorem~\ref{propabc} (here~$\mu^- \equiv 0)$ gives that~$\lambda_{\mu_1}(B_{R_2}(x_0))$ and~$\lambda_{\mu_2}(B_{R_2}(x_0))$ are both 
positive.

We set
\[
s_1:=\sup\{\mbox{supp}(\mu_1)\}\qquad{\mbox{and}}\qquad s_2:=\inf\{\mbox{supp}(\mu_2)\}.
\]
By hypothesis, we know that~$s_1<s_2$.
We point out that~$\mbox{supp}(\mu_1)\subseteq [0,s_1]$ and~$\mbox{supp}(\mu_2)\subseteq [s_2,1]$. Then, by~\eqref{zserfcvgyuhbnjiolkm} and Proposition~\ref{scalingprimoautovalore}
we deduce that
\begin{equation}\label{disuguaglianzaduemisure1}
\min\{1,r^{-2s_1}\}\lambda_{\mu_1}(B_{R_2}(x_0))\leq \lambda_{\mu_1}(\Omega_r)
\leq \max\{1,r^{-2s_1}\}\lambda_{\mu_1}(B_{R_1}(x_0))
\end{equation}
and
\begin{equation}\label{disuguaglianzaduemisure2}
\min\{r^{-2s_2},r^{-2}\}\lambda_{\mu_2}(B_{R_2}(x_0))\leq \lambda_{\mu_2}(\Omega_r)
\leq \max\{r^{-2s_2},r^{-2}\}\lambda_{\mu_2}(B_{R_1}(x_0)).
\end{equation}

We also set
\[
\underline{r}:=\min\left\{1,\left(\frac{\lambda_{\mu_2}(B_{R_2}(x_0))}{\lambda_{\mu_1}(B_{R_1}(x_0))}\right)^{\frac{1}{2(s_2-s_1)}}\right\}.
\]
Then, by~\eqref{disuguaglianzaduemisure1} and~\eqref{disuguaglianzaduemisure2} we infer that,
for any~$r\in (0,\underline{r})$,
\[
\begin{aligned}
\lambda_{\mu_2}(\Omega_r)-\lambda_{\mu_1}(\Omega_r)&
\geq \min\{r^{-2s_2},r^{-2}\}\lambda_{\mu_2}(B_{R_2}(x_0))-\max\{1,r^{-2s_1}\}\lambda_{\mu_1}(B_{R_1}(x_0)) \\
&=r^{-2s_2}\lambda_{\mu_2}(B_{R_2}(x_0)) -r^{-2s_1}\lambda_{\mu_1}(B_{R_1}(x_0)) \\
&=r^{-2s_2}\Big(\lambda_{\mu_2}(B_{R_2}(x_0)) -r^{2(s_2-s_1)}\lambda_{\mu_1}(B_{R_1}(x_0)) \Big)
\\
&>0.
\end{aligned}
\]
Thus, we can take~$\tau_r \ge 0$, $\nu_r\in L^\infty(\Omega_r)$ with~$\inf_\Omega \nu_r>0$
and~$\sigma_r\in L^\infty(\Omega_r)$, such that
\begin{equation}\label{zsrfcvgyhbnjikm,lpxdrtyhbnm}
\lambda_{\mu_1}(\Omega_r)<\inf_{\Omega_r} \sigma_r\leq \sup_{\Omega_r} \sigma_r +\tau_r\leq \lambda_{\mu_2}(\Omega_r).
\end{equation}
We observe that, with this choice, assumption~\eqref{ipotesilogistica1} is satisfied. Then, by~\eqref{zsrfcvgyhbnjikm,lpxdrtyhbnm} and Theorem~\ref{thmlogistica2} we infer that problem~\eqref{problemaduemisure1} admits a nontrivial solution, while problem~\eqref{problemaduemisure2} only admits the solution~$u\equiv 0$. This establishes the first claim in Theorem~\ref{thmduemisure}.

Now, we set 
\[
\overline{r}:=\max\left\{1,\left(\frac{\lambda_{\mu_2}(B_{R_1}(x_0))}{\lambda_{\mu_1}(B_{R_2}(x_0))}\right)^{\frac{1}{2(s_2-s_1)}}\right\}
\]
and we notice that~$\overline{r}\ge\underline{r}$.

Moreover, from~\eqref{disuguaglianzaduemisure1} and~\eqref{disuguaglianzaduemisure2} we have that,
for any~$r\in (\overline{r},+\infty)$,
\[
\begin{aligned}
\lambda_{\mu_2}(\Omega_r)-\lambda_{\mu_1}(\Omega_r)&
\leq \max\{r^{-2s_2},r^{-2}\}\lambda_{\mu_2}(B_{R_1}(x_0))-\min\{1,r^{-2s_1}\}\lambda_{\mu_1}(B_{R_2}(x_0)) \\
&=r^{-2s_2}\lambda_{\mu_2}(B_{R_1}(x_0)) -r^{-2s_1}\lambda_{\mu_1}(B_{R_2}(x_0)) \\
&=r^{-2s_2}\Big(\lambda_{\mu_2}(B_{R_1}(x_0)) -r^{2(s_2-s_1)}\lambda_{\mu_1}(B_{R_2}(x_0)) \Big)
\\&<0.
\end{aligned}
\]
Hence, we can take~$\tau_r \ge 0$, $\nu_r\in L^\infty(\Omega_r)$ with~$\inf_\Omega \nu_r>0$
and~$\sigma_r\in L^\infty(\Omega_r)$, such that
\begin{equation}\label{zsrfcvgyhbnjikm,lpxdrtyhbnm2}
\lambda_{\mu_2}(\Omega_r)<\inf_{\Omega_r} \sigma_r\leq \sup_{\Omega_r} \sigma_r +\tau_r\leq \lambda_{\mu_1}(\Omega_r).
\end{equation}
Then, assumption~\eqref{ipotesilogistica1} is satisfied. Thus, from~\eqref{zsrfcvgyhbnjikm,lpxdrtyhbnm2} 
and Theorem~\ref{thmlogistica2} we infer that problem~\eqref{problemaduemisure1} only admits the solution~$u\equiv 0$, while problem~\eqref{problemaduemisure2} admits a nontrivial solution. 
This completes the proof of the second claim in Theorem~\ref{thmduemisure}.
\end{proof}

\section{Conclusion}

In this paper, we have analyzed a diffusive logistic equation of Fisher-Kolmogoroff-Petrovski-Piskunov type in a hostile environment accounting for the abundance of natural enemies surrounding an ecological niche. The logistic component of the equation deals with population competition over environmental resources, possibly incorporating an additional pollination factor. The diffusion process is modeled as a superposition of local and nonlocal operators of varying orders, capturing the heterogeneity in dispersal patterns among individuals. These patterns include Gaussian walks and L\'evy flights with distinct L\'evy exponents.

Moreover, the diffusive terms corresponding to low L\'evy exponents are allowed to exhibit opposite signs relative to standard dispersal terms. This models concentration phenomena, where individuals migrate from areas of low to high population densities, driven by factors such as mutual protection or thermal advantages.

Our results focus on the existence and nonexistence of nontrivial
equilibrium solutions, corresponding to the survival or extinction of the population. In particular, Theorem~\ref{thmlogistica2} establishes resource thresholds required for a species to survive in hostile environments, relying on the analysis of the linearized equation and its principal eigenvalue problem, as detailed in Theorem~\ref{propabc}.

Additionally, we examine the influence of the concentrating component of diffusion on survival. Theorem~\ref{thmlogistica3} provides examples where this component enhances survival, even under otherwise adverse conditions. We show that both classical and anomalous diffusion can lead to extinction, whereas the introduction of even a small concentration pattern can enable survival. These findings underscore the biological significance of concentration phenomena in protecting populations from lethal zones.

We also explore the role of anomalous diffusion in survival dynamics, both with and without concentration effects. In Theorem~\ref{thmlogistica4}, we present scenarios where the safe niche consists of two disjoint regions, which are insufficient for survival when isolated, but are capable of supporting survival when connected via nonlocal dispersal. This highlights the beneficial role of L\'evy flights in specific biological contexts.

Finally, Theorem~\ref{thmduemisure} connects the diffusive properties of the population to the size of the survival niche. For small niches, lower diffusive exponents confer a survival advantage, while for larger niches, higher diffusive exponents become more favorable.

These results contribute to a deeper understanding of how dispersal mechanisms and environmental structures interact to influence species survival, offering insights into the complex dynamics of ecological systems.

\begin{appendix}

\section{On the assumptions~\eqref{mu2} and~\eqref{mu2forte}}\label{appendixesempio}

Here we provide an explicit calculation to show how
assumptions~\eqref{mu2} and~\eqref{mu2forte} can be checked in a
concrete case. 

For this, we let~$\alpha\ge 0$ and~$1>s_1>s_2>0$, and we set 
\begin{equation*}
\mu:=\delta_1 + \delta_{s_1}-\alpha\delta_{s_2}. 
\end{equation*}
We will check that, if~$\alpha$ is sufficiently small, the assumptions in~\eqref{mu2} and~\eqref{mu2forte} are fulfilled.

We first observe that the assumptions in~\eqref{mu0} and~\eqref{mu1} hold true with~$\overline s := s_1$.
Moreover,
\[
\mu^-\big( [0,s_1)\big)= \alpha = \frac{\alpha}{2} \mu^+\big([s_1, 1]\big),
\]
which shows that~\eqref{mu2} is satisfied for any~$\gamma\ge{\alpha}/{2}$.

We now check that~\eqref{mu2forte} is also satisfied. To this end, we let
\begin{equation*}
\overline\gamma\in\left[0, \frac{ \underline\Gamma_N \,s_1}{4\overline{\Gamma}_N\,\max\{1,(2R)^2\}}\right)
\end{equation*}
with~$\overline{\Gamma}_N$ and~$\underline{\Gamma}_N$
as in~\eqref{gammas}.

Also, if~$ \alpha\in\left[0,\overline\gamma (1-s_1)\right]$, we can pick~$ \delta:=\alpha/{\overline\gamma} $
and obtain that
\[
\mu^-\big( [0,s_1)\big)= \alpha = \alpha \mu^+\big([s_1, 1-\delta]\big)
=\overline\gamma\, \delta\, \mu^+\big([s_1, 1-\delta]\big),
\]
which proves that~\eqref{mu2forte} is satisfied.

\end{appendix}

\section*{Acknowledgements} 
All the authors are members of the Australian Mathematical Society (AustMS). CS and EPL are members of the INdAM--GNAMPA.

This work has been supported by the Australian Laureate Fellowship FL190100081
and by the Australian Future Fellowship FT230100333.

\vfill

\end{document}